\documentclass[10pt,a4paper,reqno]{article}
\usepackage{graphicx}
\usepackage[english]{babel} 

\usepackage{csquotes} 
\usepackage{filecontents} 
\usepackage[dvipsnames]{xcolor}
\usepackage{hyperref}

\usepackage[
  backend=biber,
  hyperref=true,
  backref=true,
  isbn=false,
  doi=true,
  url=true,
  natbib=true,
  eprint=true,
  useprefix=true,
  maxcitenames=99,
  maxbibnames=99,  
  maxalphanames=99, 
  minalphanames=99,
  safeinputenc,
  style=alphabetic,
  citestyle=alphabetic,
  block=space,
  datamodel=ext-eprint,
]{biblatex}
\hypersetup{
  hypertexnames = false,
  colorlinks    = true,
  citecolor     = blue,
  linkcolor     = blue,
  urlcolor      = {blue!70!black},
  linktocpage = true,
  breaklinks
  }
\renewbibmacro{in:}{} 

\DeclareSourcemap{
  \maps[datatype=bibtex]{
    \map{
      \step[fieldsource=pmid, fieldtarget=pubmed]
    }
  }
}

\makeatletter
\DeclareFieldFormat{arxiv}{%
  arXiv\addcolon\space
  \ifhyperref
    {\href{http://arxiv.org/\abx@arxivpath/#1}{%
       \nolinkurl{#1}%
       \iffieldundef{arxivclass}
         {}
         {\addspace\texttt{\mkbibbrackets{\thefield{arxivclass}}}}}}
    {\nolinkurl{#1}
     \iffieldundef{arxivclass}
       {}
       {\addspace\texttt{\mkbibbrackets{\thefield{arxivclass}}}}}}
\makeatother
\DeclareFieldFormat{pmcid}{%
  PMCID\addcolon\space
  \ifhyperref
    {\href{http://www.ncbi.nlm.nih.gov/pmc/articles/#1}{\nolinkurl{#1}}}
    {\nolinkurl{#1}}}
\DeclareFieldFormat{mrnumber}{%
  MR\addcolon\space
  \ifhyperref
    {\href{http://www.ams.org/mathscinet-getitem?mr=MR#1}{\nolinkurl{#1}}}
    {\nolinkurl{#1}}}
\DeclareFieldFormat{zbl}{%
  Zbl\addcolon\space
  \ifhyperref
    {\href{http://zbmath.org/?q=an:#1}{\nolinkurl{#1}}}
    {\nolinkurl{#1}}}
\DeclareFieldAlias{jstor}{eprint:jstor}
\DeclareFieldAlias{hdl}{eprint:hdl}
\DeclareFieldAlias{pubmed}{eprint:pubmed}
\DeclareFieldAlias{googlebooks}{eprint:googlebooks}

\renewbibmacro*{eprint}{%
  \printfield{arxiv}%
  \newunit\newblock
  \printfield{jstor}%
  \newunit\newblock
  \printfield{mrnumber}%
  \newunit\newblock
  \printfield{zbl}%
  \newunit\newblock
  \printfield{hdl}%
  \newunit\newblock
  \printfield{pubmed}%
  \newunit\newblock
  \printfield{pmcid}%
  \newunit\newblock
  \printfield{googlebooks}%
  \newunit\newblock
  \iffieldundef{eprinttype}
    {\printfield{eprint}}
    {\printfield[eprint:\strfield{eprinttype}]{eprint}}}

\usepackage[margin=2cm, top=2cm, includehead]{geometry}
\usepackage{amssymb}

\usepackage{amsmath}
\usepackage{latexsym}
\usepackage{amscd}
\usepackage{amsthm}
\usepackage{mathrsfs}
\usepackage{url}
\usepackage{amsmath,mathtools}
\usepackage{cleveref}
\usepackage{doi}

\usepackage[T1]{fontenc}
\usepackage{tikz-cd}
\usepackage{titlesec}
\usepackage[titletoc,toc,title]{appendix}
\usepackage{enumerate}
\usepackage{slashed}
\graphicspath{}
\usepackage{fancyhdr} 
\usepackage{url}

\newcommand\shorttitle{Deformation theory of nearly $\G2$ manifolds} 
\newcommand\authors{Shubham Dwivedi and Ragini Singhal} 

\newcounter{commentCounter}
\titleformat{\section}    
       {\normalfont\large\bfseries\center}{\thesection.}{1em}{}
       
\makeatletter
\newcommand*{\rom}[1]{\expandafter\@slowromancap\romannumeral #1@}
\makeatother
\newtheorem{theorem}{Theorem}[section]
\newtheorem{corollary}[theorem]{Corollary}
\newtheorem{lemma}[theorem]{Lemma}
\newtheorem{proposition}[theorem]{Proposition}

\theoremstyle{definition}
\newtheorem{definition}[theorem]{Definition}
\newtheorem{remark}[theorem]{Remark}
\newtheorem*{ack}{Acknowledgments}
\numberwithin{equation}{section}
\def\bR{\mathbb R}
\def\bZ{\mathbb Z}
\def\bC{\mathbb C}

\def\m{\mathfrak{m}}

\def\pt{\partial}
\def\del{\nabla}
\def\G2{\mathrm{G}_2}
\def\g2{\varphi}

\def\K{\mathcal{K}}
\def\C{\mathbb C}
\def\fg{\mathfrak{g}}

\def\fg2{\mathfrak{g}_{2}}
\def\fp{\mathfrak{p}}

\def\cL{\mathcal{L}}

\def\cP{\mathcal{P}}
\def\cO{\mathcal{O}}
\def\cC{\mathcal{C}}

\def\Spin{\mathrm{Spin}}
\def\Ric{\mathrm{Ric}}
\def\Rm{\mathrm{Rm}}
\def\SO{\mathrm{SO}}
\def\GL{\mathrm{GL}}

\def\dirac{\slashed{D}}

\DeclareMathOperator\Div{div}
\DeclareMathOperator\curl{curl}
\DeclareMathOperator\grad{grad}
\DeclareMathOperator\vol{vol}

\DeclareMathOperator\tr{tr}
\DeclareMathOperator{\Ima}{Im}
\DeclareMathOperator{\exact}{\textup{exact}}

\newcommand\blfootnote[1]{%
  \begingroup
  \renewcommand\thefootnote{}\footnote{#1}%
  \addtocounter{footnote}{-1}%
  \endgroup
}

\begin{document}

\title{Deformation theory of nearly $\G2$ manifolds}
\author{Shubham Dwivedi \hspace{2cm} Ragini Singhal}
\date{\today}

\maketitle

%

\begin{abstract}
We study the deformation theory of nearly $\G2$ manifolds. These are seven dimensional manifolds admitting real Killing spinors. We show that the infinitesimal deformations of nearly $\G2$ structures are obstructed in general. Explicitly, we prove that the infinitesimal deformations of the homogeneous nearly $\G2$ structure on the Aloff--Wallach space are all obstructed to second order. We also completely describe the cohomology of nearly $\G2$ manifolds.   
\end{abstract}


\tableofcontents

\section{Introduction}\label{sec:intro}

Given a $7$-dimensional smooth manifold $M$, a nearly $\G2$ structure on $M$ is a non-degenerate (or positive) $3$-form $\g2$ such that for some non-zero real constant $\tau_0$,
\begin{align}\label{eq:ng2relnaux}
d\g2=\tau_0*_{\g2}\g2    
\end{align}
where the metric and the orientation and hence the Hodge star $*$ are all induced by $\g2$. The existence of a nearly $\G2$ structure was shown to be equivalent to the existence of a \emph{real Killing spinor} in \cite{baum-etal}. A Killing spinor on a Riemannian spin manifold $(M^n,g)$ is a section of the spinor bundle $\mu \in \Gamma(\slashed{S}(M))$ such that
\begin{align}\label{killspin1}
\del_X \mu = \alpha X\cdot \mu    
\end{align}
for any vector field $X$ on $M$ and some $\alpha \in \bC$. Here $\cdot$ is the Clifford multiplication. It was proved by Friedrich \cite{friedrich} that any manifold with a Killing spinor is Einstein with $\Ric(g)=4(n-1)\alpha^2g$ and one of the three cases must hold:
\begin{itemize}
\item $\alpha=0$ in which case $\mu$ is a parallel spinor and $M$ has holonomy contained in $\mathrm{SU}(\frac n2)$, $\mathrm{Sp}(\frac n4)$, $\G2$ or $\mathrm{Spin}(7)$.

\item $\alpha$ is non-zero and is purely imaginary.

\item $\alpha$ is non-zero and real, in which case $\mu$ is a real Killing spinor and if $M$  is complete then since it is positive Einstein, it is compact with $\pi_1(M)$ finite.  
\end{itemize}
Given a nearly $\G2$ structure $\g2$ on $M$ that satisfies equation \eqref{eq:ng2relnaux}, there exists a real Killing spinor $\mu$ that satisfies equation \eqref{killspin1} with $\alpha=-\frac{1}{8}\tau_0$ and vice-versa. See \cite{baum-etal} for more details.

\medskip

\noindent
Using the equivalence with real Killing spinors, nearly $\G2$ structures on homogeneous spaces, excluding the case of the round $7$-sphere, were classified in \cite{f-k-m-s}. Their classification is based on the dimension of the space of Killing spinors $K\slashed{S}$. They showed that $3$ different types can occur:

\begin{enumerate}
\item dim($K\slashed{S})=1$ - nearly $\G2$ structures of type 1.

\item dim($K\slashed{S})=2$ - nearly $\G2$ structures of type 2.

\item dim($K\slashed{S})=3$ - nearly $\G2$ structures of type 3.
\end{enumerate}

\medskip

\noindent
A $7$-dimensional manifold $(M, \g2)$ with a nearly $\G2$ structure $\g2$ is a nearly $\G2$ manifold (see \textsection \ref{sec:prelims} for more details). 
Other examples apart from the round $S^7$ include the squashed $S^7$, Aloff--Wallach spaces $N(k,l)$, the Berger space $\mathrm{SO}(5)/\mathrm{SO}(3)$ and the Stiefel manifold $V_{5,2}$. Another important aspect of nearly $\G2$ manifolds is that the Riemannian cone $C(M)$ over $M$ has holonomy contained in the Lie group $\mathrm{Spin}(7)$. In that case, the possible holonomies are $\mathrm{Spin}(7)$, $\mathrm{SU}(4)$ or $\mathrm{Sp}(2)$ depending on whether the link of the cone is a nearly $\G2$ manifold of type $1,\ 2$ or $3$ respectively.

\medskip

\noindent
In this paper, we study the deformation theory of nearly $\G2$ manifolds. The infinitesimal deformations of nearly $\G2$ manifolds were studied by Alexandrov--Semmelmann in \cite{deformg2} where they identified the space of infinitesimal deformations with an eigenspace of the Laplacian acting on co-closed $3$-forms on $M$ of type $\Omega^3_{27}$. We address the question of whether nearly $\G2$ manifolds have smooth obstructed or unobstructed deformations, i.e., whether infinitesimal deformations can be integrated to genuine deformations. This could potentially give new examples of nearly $\G2$ manifolds. Another applicability of studying the deformation theory of nearly $\G2$ manifolds can be to develop the deformation theory of $\Spin(7)$ \emph{conifolds} which are asymptotically conical and conically singular $\Spin(7)$ manifolds, similar to the theory developed by Karigiannis--Lotay \cite{kargiannis-lotay} for $\G2$ conifolds. Lehmann \cite{lehmann} studies the deformation theory of asymptotically conical $\mathrm{Spin}(7)-$manifolds.

\medskip

\noindent
The study of deformation theory of special algebraic structures is not new. Deformations of Einstein metrics were studied by Koiso where he showed \cite[Theorem 6.12]{koiso} that the infinitesimal deformations of Einstein metrics is in general obstructed, by exhibiting certain Einstein symmetric spaces which admit non-trivial infinitesimal Einstein deformations which cannot be integrated to second order. The deformation theory of nearly K\"ahler structures on homogeneous $6$-manifolds was studied by Moroianu--Nagy--Semmelmann in \cite{m-n-s}. They identified the space of infinitesimal deformations with an eigenspace of the Laplacian acting on co-closed primitive $(1,1)$-forms. Using this, they proved that the nearly K\"ahler structures on $\mathbb{CP}^3$ and $S^3\times S^3$ are rigid and the flag manifold $\mathbb{F}_3$ admits an $8$-dimensional space of infinitesimal deformations. Later, Foscolo proved \cite[Theorem 5.3]{foscolo} that the infinitesimal deformations of the flag manifold $\mathbb{F}_3$ are all obstructed.

\medskip

\noindent
Nearly $\G2$ manifolds are in many ways similar to nearly K\"ahler $6$-manifolds. Both admit real Killing spinors and hence are positive Einstein. The minimal hypersurfaces in both nearly K\"ahler 6-manifolds and nealy $\G2$ manifolds behave in a similar way \cite{dwivedi-minimal}. 
It was proved in \cite{deformg2} that the nearly $\G2$ structures on the squashed $S^7$ and the Berger space $\mathrm{SO}(5)/\mathrm{SO}(3)$ are rigid while the space of infinitesimal nearly $\G2$ deformations of the Aloff--Wallach space $X_{1,1}$ is $8$-dimensional. It is therefore natural to ask whether these infinitesimal deformations are obstructed to second order.

\medskip

\noindent
To address this question, we use a Dirac-type operator on nearly $\G2$ manifolds (cf. equation \eqref{moddirac}). The use of Dirac operators to study deformation theory has been very useful. Nordstr\"om in \cite{nordstrom-thesis} used Dirac operators to study the deformation theory of compact manifolds with special holonomy from a different point of view than Joyce \cite{joycebook}. In particular, the mapping properties of the Dirac type operators can be used to prove slice theorems  for the action of the diffeomorphism group. This approach has also been very effective in studying the deformation theory of non-compact manifolds with special holonomy, most notably by Nordstr\"om \cite{nordstrom-thesis} for asymptotically
cylindrical manifolds with exceptional holonomy and by Karigiannis--Lotay \cite{kargiannis-lotay} for $\G2$ conifolds. Dirac-type operators, in a way very close to the use made by the authors in this paper, were also used by Foscolo \cite{foscolo} to study the deformation theory of nearly K\"ahler $6$-manifolds.

\medskip

\noindent
We follow a strategy similar to \cite{foscolo} in this paper. After introducing the Dirac operator and a modified Dirac operator on nearly $\G2$ manifolds in \textsection \ref{sec:hodgetheory}, we use their properties and the Hodge decomposition theorem to completely describe the cohomology of a complete nearly $\G2$ manifold. We prove our first two main results of the paper which characterize harmonic forms. These are the following. 

\medskip

\noindent
{\bf{Theorem \ref{thm:harmonic3and4form}.}} \emph{Let $(M,\g2,\psi)$ be  a complete nearly $\G2$ manifold, not isometric to round $S^7$. Then every harmonic $4$-form lies in $\Omega^4_{27}$. Equivalently, every harmonic $3$-form lies in $\Omega^3_{27}$.}

\medskip

\noindent
\noindent
{\bf{Theorem \ref{lem:2coho}}} \emph{Let $(M,\g2,\psi)$ be  a complete nearly $\G2$ manifold, not isometric to round $S^7$. Then every harmonic $2$-form lies in $\Omega^2_{14}$. Equivalently, every harmonic $5$-form lies in $\Omega^5_{14}$.}

\medskip

\noindent
We note that Theorem \ref{lem:2coho} was originally proved by Ball--Oliveira \cite[Remark 15]{ball-oliveira}. We give a different proof in this paper.

\medskip

\noindent
We use the properties of the modified Dirac operator, explicitly we use Proposition \ref{prop:4form}, to prove a slice theorem for the action of the diffeomorphism group on the space of nearly $\G2$ structures on $M$ in Proposition \ref{slicethm}. Using this, in Theorem \ref{thm:infidef} we obtain a new proof of the identification of the space of nearly $\G2$ deformations with an eigenspace of the Laplacian acting on co-closed $3$-forms of type $\Omega^3_{27}$, a result originally due to Alexandrov--Semmelmann \cite{deformg2}.

\medskip

\noindent
To study higher order deformations of nearly $\G2$ manifolds, we use the point view of Hitchin \cite{hitchin} where he interprets nearly $\G2$ structures as constrained critical points of a functional defined on the space $\Omega^3\times \Omega^4_{\exact}$. This approach is inspired from the work of Foscolo \cite{foscolo} where he used similar ideas to study second order deformations of nearly K\"ahler structures on $6$-manifolds. The advantage of this approach is that it allows us to view the nearly $\G2$ equation \eqref{eq:ng2reln} as the vanishing of a smooth map (cf. equation \eqref{zerolocus})
\begin{align*}
\Phi : \Omega^4_{+, \exact}\times \Gamma(TM)\longrightarrow \Omega^4_{\exact}    
\end{align*}
where $\Omega^4_{+, \exact}$ denotes the space of exact \emph{positive} $4$-forms on $M$. Thus the obstructions on the first order deformations of a nearly $\G2$ structure to be integrated to higher order deformations can be characterized by $\Ima(D\Phi)$ which we do in Proposition \ref{imageDphi}. 

\medskip

\noindent
Finally, we use the general deformation theory of nearly $\G2$ structures developed in the first part of the paper to study the infinitesimal deformations of the Aloff--Wallach space $\frac{\mathrm{SU}(3)\times \mathrm{SU}(2)}{\mathrm{SU}(2)\times \mathrm{U}(1)}$. It was expected in \cite{foscolo} that the infinitesimal deformations of the Aloff--Wallach space might be obstructed to higher orders. In \textsection \ref{sec:awspace} we confirm this expectation. More precisely, we prove the following.

\medskip

\noindent
{\bf{Theorem \ref{thm:awobs}.}} \emph{The infinitesimal deformations of the homogeneous nearly $\G2$ structure on the Aloff--Wallach space $X_{1,1}\cong \frac{\mathrm{SU}(3)\times\mathrm{SU}(2)}{\mathrm{SU}(2)\times\mathrm{U}(1)}$ are all obstructed.}

\medskip

\noindent
The proof of the above theorem is inspired from the ideas in \cite{foscolo}. However, we note that since in the nearly $\G2$ case we only have one \emph{stable} form and the other is the dual of it, unlike the nearly K\"ahler case, the expressions and computations involved are more complicated and the proof of the theorem is computationally much more involved.

\medskip

\noindent
The paper is organized as follows. We discuss some preliminaries on $\G2$ and nearly $\G2$ structures in \textsection \ref{sec:prelims}. We discuss the decomposition of space of differential forms on manifolds with a $\G2$ structure. We describe some first order differential operators in \textsection \ref{subsec:fodo} which appear throughout the paper. In \textsection \ref{subsec:23forms}, we prove many important identities for $2$-forms and $3$-forms on manifolds with nearly $\G2$ structures. Some of these appear to be new, at least in the present form and we believe that they will be useful in other contexts as well. We introduce the Dirac and the modified Dirac operator in \textsection \ref{sec:hodgetheory} and use the mapping properties of the latter to prove Theorem \ref{thm:harmonic3and4form} and Theorem \ref{lem:2coho}. We begin the discussion on infinitesimal deformations in \textsection \ref{subsec:infinitesimaldeform}. We prove a slice theorem and use that to obtain a new proof of the result of Alexandrov--Semmmelmann on infinitesimal nearly $\G2$ deformations. We interpret the nearly $\G2$ equation as the vanishing of a smooth map and prove the characterization for a first order deformation of a nearly $\G2$ structure to be integrated to second order in Proposition \ref{imageDphi}. Finally, in \textsection \ref{sec:awspace}, we prove Theorem \ref{thm:awobs}.

\medskip

\noindent
{\bf{Note.}} The almost simultaneous preprint \cite{semmelmann-nagy} by Semmelmann--Nagy has some overlap with the present paper and some of the ideas involved are the same. We also characterize the cohomology of nearly $\G2$ manifolds. The second version of their paper also contains a discussion of the deformations of the Aloff--Wallach spaces.

\medskip

 \begin{ack}
We are indebted to Spiro Karigiannis and Benoit Charbonneau for various discussions related to the paper and for constant encouragement and advice. We are grateful to Ben Webster for an important discussion on representation theory. We thank Gavin Ball and Gon\c{c}alo Oliveira for pointing us out to their result about harmonic $2$-forms on nearly $\G2$ manifolds in their paper \cite{ball-oliveira}. We are grateful to Gon\c{c}alo Oliveira for discussions on the material in \textsection \ref{sec:awspace}. Finally, we would like to thank the referee for a very careful reading of the paper and for many useful remarks and suggestions which have improved the quality of the paper.
 \end{ack}

\section{Preliminaries on $\G2$ geometry}\label{sec:prelims}

\medskip

We start this section by defining $\G2$ structures and nearly $\G2$ structures on a
seven dimensional manifold and also discuss the decomposition of space of differential forms on such a manifold. We also collect together various identities which will be used throughout the paper.

\medskip

\noindent
Let $M^7$ be a smooth manifold. A $\G2$ structure on M is a reduction of the structure group of the frame bundle from $\GL(7, \bR)$ to the Lie group $\G2 \subset \SO(7)$. Such a structure exists on $M$ if and only if
the manifold is orientable and spinnable, conditions which are respectively equivalent to the vanishing
of the first and second Stiefel--Whitney classes. From the point of view of differential
geometry, a $\G2$ structure on M is equivalently defined by a $3$-form $\g2$ on $M$ that satisfies a certain
pointwise algebraic non-degeneracy condition. Such a $3$-form nonlinearly induces a Riemannian metric $g_{\g2}$ and an orientation $\vol_{\g2}$ on $M$ and hence a Hodge
star operator $*_{\g2}$. We denote the Hodge dual $4$-form $*_{\g2}\g2$ by $\psi$. Pointwise we have $|\g2|=|\psi| = 7$, where the norm is taken with respect to the metric induced by $\g2$.

\medskip

\noindent
{\bf{Notations and conventions.}} Throughout the paper, we compute in a local orthonormal frame, so all indices are subscripts and any repeated indices are summed over all values from $1$ to $7$. Our convention for labelling the Riemann curvature tensor is
$$R_{ijkm} \frac{\pt}{\pt x^m}
= (\del_i\del_j-\del_j\del_i)\frac{\pt}{\pt x^k},$$in terms of coordinate vector fields. With this convention, the Ricci tensor is $R_{jk} = R_{ljkl}$, and the Ricci
identity is
\begin{align}\label{ricciidentity}
\del_i\del_jX_k-\del_j\del_iX_k = -R_{ijkl}X_l.
\end{align}
We will use the metric to identify the vector fields and $1$-forms by the musical isomorphisms. As such, throughout the paper, we will use them interchangeably without mention.

\medskip

\noindent
We have the following contraction identities between $\g2$ and $\psi$, whose proofs can be found in \cite{skflow}.
\begin{align}
\g2_{ijk}\g2_{abk}&=g_{ia}g_{jb}-g_{ib}g_{ja}+\psi_{ijab},  \label{eq:phiphi1}  \\
\g2_{ijk}\g2_{ajk}&=6g_{ia}\label{eq:phiphi2}
\end{align}
and
\begin{align}
\g2_{ijk}\psi_{abck}&=g_{ja}\g2_{ibc}+g_{jb}\g2_{aic}+g_{jc}\g2_{abi}-g_{ia}\g2_{jbc}-g_{ib}\g2_{ajc}-g_{ic}\g2_{abj}, \label{eq:phipsi1} \\
\g2_{ijk}\psi_{abjk}&=4\g2_{iab},\label{eq:phipsi2} \\
\psi_{ijkl}\psi_{abkl}&=4g_{ia}g_{jb}-4g_{ib}g_{ja}+2\psi_{ijab} \label{eq:psiwithpsi2}\\
\psi_{ijkl}\psi_{ajkl}&=24g_{ia}.\label{eq:psiwithpsi}
\end{align}

\medskip

\noindent
A $\G2$ structure on $M$ induces a splitting of the spaces of differential forms on $M$ into irreducible $\G2$ representations.  The space of $2$-forms $\Omega^2(M)$ and $3$-forms $\Omega^3(M)$ decompose as 
\begin{align}
\Omega^2(M)&=\Omega^2_7(M)\oplus \Omega^2_{14}(M), \label{eq:decomp1}\\
\Omega^3(M)&=\Omega^3_1(M)\oplus \Omega^3_7(M)\oplus \Omega^3_{27}(M) \label{eq:decomp2}
\end{align}
where $\Omega^k_l$ has pointwise dimension $l$. More precisely, we have the following description of the space of forms :

 \begin{align} 
 \Omega^2_7(M) &=\{X\lrcorner \g2\mid X\in \Gamma(TM)\} = \{\beta \in \Omega^2(M)\mid *(\g2\wedge \beta)=2\beta\} \label{2formsdecomposition7}, \\
 \Omega^2_{14}(M) &=\{\beta \in \Omega^2(M)\mid \beta \wedge \psi =0 \} = \{\beta\in \Omega^2(M)\mid *(\g2\wedge \beta)=-\beta\}. \label{2formsdecomposition14}
 \end{align}
In local coordinates, the above conditions can be re-written as
\begin{align}
\beta\in \Omega^2_7 \ \ \ &\iff\ \ \ \beta_{ij}\psi_{abij}=4\beta_{ab}, \label{eq:27decomp1} \\
\beta \in \Omega^2_{14}\ \ &\ \iff\ \ \ \beta_{ij}\psi_{abij} = -2\beta_{ab}\ \ \ \iff\ \ \ \beta_{ij}\g2_{ijk}=0. \label{eq:214decomp1}
\end{align}
Similarly, for $3$-forms
\begin{align}
\Omega^3_1 &=\{ f\g2 \mid f\in C^{\infty}(M)\}, \label{eq:3formdecom1} \\
\Omega^3_7 & = \{ X\lrcorner \psi \mid X\in \Gamma(TM)\} = \{*(\alpha \wedge \g2) \mid \alpha \in \Omega^1\}, \label{eq:3formdecom2}\\
\Omega^3_{27} & = \{ \eta \in \Omega^3 \ \mid \ \eta\wedge \g2 = 0 = \eta\wedge \psi\}. \label{eq:3formdecom3}
\end{align}
Moreover, the space $\Omega^3_{27}$ is isomorphic to the space of sections of $S^2_0(T^*M)$, the traceless symmetric $2$-tensors on M, where the isomorphism $i_{\g2}$ is given explicitly as

\begin{equation}\label{eq:327express}
 \begin{aligned}
\eta = \frac 16\eta_{ijk}dx^i\wedge dx^j\wedge dx^k \in \Omega^3_{27} \ \ \ \ \ \ \ \overset{i_{\g2}}{\longleftrightarrow}\ \ \ \ \ \ \  h_{ab}dx^adx^b \in C^{\infty}(S^2_0(T^*M))\\
\textup{where}\ \ \ \ \ \eta_{ijk} = h_{ip}\g2_{pjk}+h_{jp}\g2_{ipk}+h_{kp}\g2_{ijp}.
  \end{aligned}
\end{equation} 
The decompositions of $\Omega^4(M)=\Omega^4_1(M)\oplus \Omega^4_7(M)\oplus \Omega^4_{27}(M)$ and $\Omega^5(M)=\Omega^5_7(M)\oplus \Omega^5_{14}(M)$ are obtained by taking the Hodge star of \eqref{eq:decomp2} and \eqref{eq:decomp1} respectively. 

\medskip

\noindent
Given a $\G2$ structure $\g2$ on $M$, we can decompose $d\g2$ and $d\psi$ according to \eqref{eq:decomp1} and \eqref{eq:decomp2}. This defines the \emph{torsion forms}, which are unique differential forms $\tau_0 \in \Omega^0(M)$, $\tau_1 \in \Omega^1(M)$, $\tau_2 \in \Omega^2_{14}(M)$ and $\tau_3 \in \Omega^3_{27}(M)$ such that (see \cite{skflow})

\begin{align}
d\g2 &= \tau_0\psi + 3\tau_1\wedge \g2 + *_{\g2}\tau_3,  \label{torsionforms1} \\
d\psi &= 4\tau_1\wedge \psi + *_{\g2} \tau_2. \label{torsionforms2}
\end{align}

\noindent
Let $\del$ denote the Levi-Civita connection of the metric induced by the $\G2$ structure. The \emph{full torsion tensor} $T$ of a $\G2$ structure is a $2$-tensor satisfying 

\begin{align}
\del_i\g2_{jkl} &= T_{im}\psi_{mjkl}, \label{eq:delphi} \\
T_{lm} &=\frac {1}{24}(\del_l\g2_{abc})\psi_{mabc}, \label{torsion2}\\
\del_m\psi_{ijkl} &= -T_{mi}\g2_{jkl}+T_{mj}\g2_{ikl}-T_{mk}\g2_{ijl}+T_{ml}\g2_{ijk}. \label{eq:delpsi}
\end{align}
The full torsion $T$ is related to the torsion forms by (see \cite{skflow}) 

\begin{equation}\label{torsionrel}
T_{lm}=\frac{\tau_0}{4}g_{lm}-(\tau_3)_{lm}-(\tau_1)_{lm}-\frac{1}{2}(\tau_2)_{lm}.
\end{equation}

\begin{remark}
The space $\Omega^2_{7}$ is isomorphic to the space of vector fields and hence to the space of $1$-forms. Thus in \eqref{torsionrel}, we are viewing $\tau_1$ as an element of $\Omega^2_{7}$ which justifies the expression $(\tau_1)_{lm}$. 
\end{remark}

\medskip

\noindent
A $\G2$ structure $\g2$ is called {\bf{torsion-free}} if $\del \g2 =0$ or equivalently $T=0$. We can now define nearly $\G2$ structures.


\begin{definition}\label{nearlyg2defn}
A $\G2$ structure $\g2$ is a {\bf{nearly $\G2$}} structure if $\tau_0$ is the only nonvanishing component of the torsion, that is  
\begin{align}\label{eq:ng2reln}
d\g2=\tau_0\psi \ \ \ \ \textup{and} \ \ \ \ d\psi=0.
\end{align} 
\end{definition}

\noindent
In this case, we see from \eqref{torsionrel} that $T_{ij}=\dfrac{\tau_0}{4}g_{ij}$. 

\begin{remark}
If $\g2$ is a nearly $\G2$ structure on $M$ then since $d\g2=\tau_0\psi$, we can differentiate this to get $d\tau_0\wedge \psi =0$ and hence $d\tau_0 =0$, as wedge product with $\psi$ is an isomorphism from $\Omega^1_7(M)$ to $\Omega^5_7(M)$. Thus $\tau_0$ is a constant, if $M$ is connected.
\end{remark}

\noindent
Given a $\G2$ structure $\g2$ with torsion $T_{lm}$, we have the expressions for the Ricci curvature $R_{ij}$ and the scalar curvature $R$ of its associated metric $g$ which can be found in \cite{bryantrmks} or \cite{skflow} as 

\begin{align}
R_{jk}&=(\del_iT_{jm}-\del_jT_{im})\g2_{mki} - T_{jl}T_{lk}+ \tr(T)T_{jk}-T_{jb}T_{lp}\psi_{lpbk}, \label{ricci} \\
R&=-12\del_i(\tau_1)_i+\frac{21}{8}{\tau_0}^2-|\tau_3|^2+5|\tau_1|^2-\frac 14 |\tau_2|^2. \label{scalar}
\end{align}
\noindent
where $|C|^2=C_{ij}C_{kl}g^{ik}g^{jl}$ is the matrix norm in \eqref{scalar}.

\vspace{0.2cm}

\noindent
In particular, for a manifold $M$ with a nearly $\G2$ structure $\g2$, we see that 
\begin{align}
R_{ij} &= \frac{3}{8}{\tau_0}^2g_{ij}, \label{nearlyricci}\\
R &=\frac{21}{8}{\tau_0}^2. \label{nearlyscalar}
\end{align}

\noindent
Finally, we remark that $S^7$ with the round metric and also the squashed $S^7$ are examples of manifolds with nearly $\G2$ structure (see \cite{f-k-m-s} for more on nearly $\G2$ structures. The authors in \cite{f-k-m-s} call such structures nearly parallel $\G2$ structures but we will call them nearly $\G2$ structures.) In particular, $S^7$ with radius $1$ has scalar curvature $42$, so comparing with \eqref{eq:ng2reln} we get that $\tau_0 = 4$.

\medskip

\noindent
We use the following identities throughout the paper. They are all proved in \cite[Lemma 2.2.1 and Lemma 2.2.3]{karigiannis-deformations} and we collect them here for the convenience of the reader. First, we note that if $\alpha$ is a $k$-form and $w$ is a vector field then
\begin{align}
*(w\lrcorner \alpha)&=(-1)^{k+1}(w\wedge *\alpha), \label{eq:impiden1}\\
*(w\wedge \alpha)&=(-1)^k(w\lrcorner *\alpha).\label{eq:impiden2}
\end{align}
If $\alpha$ is a $1$-form then we have the following identities
\begin{align}
*(\g2\wedge *(\g2 \wedge \alpha))&=-4\alpha, \label{eq:impiden3} \\
\psi \wedge *(\g2 \wedge \alpha)&=0, \label{eq:impiden4} \\
*(\psi \wedge *(\psi \wedge \alpha))&=3\alpha, \label{eq:impiden5} \\
\g2\wedge *(\psi \wedge \alpha)&=2(\psi \wedge \alpha). \label{eq:impiden6}
\end{align}

\medskip

\noindent
Suppose $w$ is a vector field then we have the following identities
\begin{align}
\g2\wedge (w\lrcorner \psi)&=-4*w, \label{eq:impiden7}\\
\psi \wedge (w\lrcorner \psi)&=0, \label{eq:impiden8}\\
\psi \wedge (w\lrcorner \g2)&= 3*w, \label{eq:impiden9}\\
\g2\wedge (w\lrcorner \g2)&= 2*(w\lrcorner \g2). \label{eq:impiden10}
\end{align}

\medskip

\noindent
Let $\Theta : \Omega^3_+\rightarrow \Omega^4_+$ be the non-linear map which associates to any $\G2$ structure $\g2$, the dual $4$-form $\psi=\Theta(\g2)=*\g2$ with respect to the metric $g_{\g2}$. We note that $\Theta^{-1}: \Omega^4_{+}\rightarrow \Omega^3_{+}$ is defined only when we fix the orientation on $M$. See \cite[\textsection 8]{hitchin} for more details. We will need the following result from \cite[Proposition 10.3.5]{joycebook}, later.

\begin{proposition}\label{prop:linearizationmap}
Suppose $\g2$ be a $\G2$ structure on $M$ with $\psi=*\g2$. Let $\xi$ be a $3$-form which has sufficiently small pointwise norm with respect to $g_{\g2}$ so that $\g2+\xi$ is still a positive $3$-form and $\eta$ be a $4$-form with small enough pointwise norm so that $\psi+\eta$ is a positive $4$-form. Then 
\begin{enumerate}[(1)]
\item the image of $\xi$ under the linearization of $\Theta$ at $\g2$ is 
\begin{align}\label{eq:linearizationmap1}
\Theta(\xi) = *_{\g2}\Big(\frac 43\pi_1(\xi)+\pi_7(\xi)-\pi_{27}(\xi)  \Big).
\end{align}
\item the image of $\eta$ under the linearization of $\Theta^{-1}$ at $\psi$ is
\begin{align}\label{eq:linearizationmap2}
\Theta^{-1}(\eta) = *_{\g2}\Big(\frac 34\pi_1(\eta)+\pi_7(\eta)-\pi_{27}(\eta)  \Big).
\end{align}
\end{enumerate}
\end{proposition}

\subsection{First order differential operators}\label{subsec:fodo}

\noindent
In this section, we discuss various first order differential operators on a manifold with a nearly $\G2$ structure and prove some identities involving them.

\medskip

\noindent
For $f\in C^{\infty}(M)$, we have the vector field $\grad f$ given by $$(\grad f)_k = \del_kf$$and for any vector field $X$ we have the divergence of $X$ which is a function $$\Div X = \del_kX_k.$$On a manifold with a $\G2$ structure $\g2$, for a vector field $X\in \Gamma(TM)$, we define the \emph{curl} of $X$, as 
\begin{align}\label{eq:curl1}
(\curl X)_k=\del_iX_j\g2_{ijk}
\end{align}
which can also be written as
\begin{align}\label{eq:curl2}
(\curl X) = *(dX\wedge \psi)
\end{align} 
and so up to $\G2$-equivariant isomorphisms, the vector field $\curl X$ is the projection of the $2$-form $dX$ onto the $\Omega^2_7$ component. In fact, we have the following 

\medskip

\begin{proposition}\label{prop:curlxdx}
Let $X$ be a vector field on $M$. The $\Omega^2_7$ component of $dX$ is given by
\begin{align}\label{eq:curlxdx}
\pi_7(dX)=\frac 13(\curl X)\lrcorner \g2 = \frac 13*(\curl X\wedge \psi).
\end{align}
\end{proposition}
\begin{proof}
We know that $\pi_7(dX)=W\lrcorner \g2$ for some vector field $W$. Using \eqref{eq:impiden9} we compute
\begin{align*}
\curl X &= *(dX\wedge \psi) = *(\pi_7(dX)\wedge \psi) = *((W\lrcorner \g2)\wedge \psi)=3W
\end{align*}
which gives \eqref{eq:curlxdx}.
\end{proof}
\medskip

\noindent
In the next proposition we state and prove various relations among the first order differential operators described above. We prove the results for \emph{any} $\G2$ structure and will later state the results for nearly $\G2$ structures. These formulas are generalizations of the formulas first proved for torsion-free $\G2$ structures by Karigiannis \cite[Proposition 4.4]{karigiannis-notes}. 
\begin{proposition}\label{prop:firstorderdiff}
Let $f\in C^{\infty}(M)$ and $X$ be a vector field on $M$ with a $\G2$ structure $\g2$. Then
\begin{align}
\curl(\grad f)& =0, \label{eq:curlgrad} \\
\Div(\curl X)&=\del_iX_j(4(\tau_1)_{ij}-(\tau_2)_{ij}) + (\pi_7(\Rm))_{jl}^jX_l, \label{eq:divcurl} \\ 
\curl(\curl X)_l&= \del_l(\Div X)+R_{lm}X_m-\Delta X_l-(\curl X)_mT_{ml}-(\del_lX_i-\del_iX_l)(\tau_1)_{ms}\g2_{msi} \nonumber \\
& \quad +\tr T(\curl X)_l+\del_iX_jT_{is}\g2_{jsl}+\del_iX_jT_{js}\g2_{sil}. \label{eq:curlcurl}
\end{align}
\end{proposition}

\begin{remark}\label{rem:pi7rm}
For fixed $i,\ j$, the Riemann curvature tensor $R_{ijkl}$ is skew-symmetric in $k$ and $l$ and hence
\begin{align*}
R_{ijkl}=(\pi_7(\Rm))_{ijkl}+(\pi_{14}(\Rm))_{ijkl}.    
\end{align*}
Explicitly,
\begin{align*}
(\pi_7(\Rm))_{ijkl}=\frac 13 R_{ijkl}+\frac 16 R_{abkl}\psi_{abij},\ \ \ (\pi_{14}(\Rm))_{ijkl}=\frac 23R_{ijkl}-\frac 16R_{abkl}\psi_{abij}.    
\end{align*}
Moreover, from \cite[eq. (4.17)]{skflow}, we have
\begin{align}\label{aux1}
(\pi_7(\Rm))_{ijkl}=(\pi_7(\Rm))^m_{ij}\g2_{mkl}\ \ \ \ \  \textup{where}\ \ \ \ \ \pi_7(\Rm)^m_{ij}=\frac 16 R_{ijkl}\g2_{klm}.    
\end{align}
\end{remark}

\begin{proof}
We compute
\begin{align*}
\curl(\grad f)&= \del_i(\del_jf)\g2_{ijk}=0
\end{align*}
as $\g2$ is skew-symmetric, thus proving \eqref{eq:curlgrad}. For \eqref{eq:divcurl} we use the Ricci identity \eqref{ricciidentity} to get
\begin{align*}
\Div(\curl X) &= \del_k(\del_iX_j\g2_{ijk})\\
&=\del_k\del_iX_j\g2_{ijk}+\del_iX_j\del_k\g2_{ijk}\\
&=\frac 12(\del_k\del_iX_j-\del_i\del_kX_j)\g2_{ijk}+\del_iX_jT_{km}\psi_{mijk}\\
&=-\frac 12R_{kijl}X_l\g2_{ijk}+\del_iX_j(4(\tau_1)_{ij}-(\tau_{2})_{ij})\\
&=3(\pi_7(\Rm))_{lj}^jX_l + \del_iX_j(4(\tau_1)_{ij}-(\tau_{2})_{ij}) 
\end{align*}
where we used \eqref{eq:27decomp1}, \eqref{eq:214decomp1} and \eqref{aux1}. We have also used the fact that the symmetric part of $T$ will vanish when contracted with $\psi$.

\medskip

\noindent
Finally we use the contraction identities \eqref{eq:phiphi1} and \eqref{eq:phipsi1} and the Ricci identity \eqref{ricciidentity} to compute
\begin{align*}
(\curl(\curl X))_l&= \del_m(\del_iX_j\g2_{ijk})\g2_{mkl} \\
&= (\del_m\del_iX_j \g2_{ijk}+\del_iX_jT_{ms}\psi_{sijk})\g2_{lmk}\\
&=\del_m\del_iX_j(g_{il}g_{jm}-g_{im}g_{jl}+\psi_{ijlm})\\
& \quad +\del_iX_jT_{ms}(g_{ms}\g2_{lij}+g_{mi}\g2_{slj}+g_{mj}\g2_{sil}-g_{ls}\g2_{mij}-g_{li}\g2_{smj}-g_{lj}\g2_{sim})\\
&=\del_j\del_lX_j-\Delta X_l + \frac{1}{2}(\del_m\del_iX_j-\del_i\del_mX_j)\psi_{ijlm}+\tr T \del_iX_j\g2_{ijl}+\del_iX_jT_{is}\g2_{slj} \\
& \quad +\del_iX_mT_{ms}\g2_{sil} -\del_iX_jT_{ml}\g2_{mij}-\del_lX_jT_{ms}\g2_{smj}-\del_iX_lT_{ms}\g2_{msi}  \\
& = \del_l(\Div X)+R_{lm}X_m-\Delta X_l+\tr T(\curl X)_l + \del_iX_jT_{is}\g2_{jsl}+ \del_iX_mT_{ms}\g2_{sil} \\
& \quad - (\curl X)_mT_{ml}-\del_lX_j(\tau_1)_{ms}\g2_{msj}+\del_iX_l(\tau_1)_{ms}\g2_{msi}\\
\end{align*}
where we used the fact that $R_{abcd}\psi_{abck}=0$ for the third term in the fourth equality and \eqref{eq:214decomp1} to cancel the $\tau_2$ components which contract on two indices with $\g2$ for the last two terms in the fourth equality. Thus, we get
\begin{align*}
(\curl(\curl X))_l&= \del_l(\Div X)+R_{lm}X_m-\Delta X_l-(\curl X)_mT_{ml}-(\del_iX_l-\del_lX_i)(\tau_1)_{ms}\g2_{msi}\\
& \quad +\tr T(\curl X)_l+\del_iX_jT_{is}\g2_{jsl}+\del_iX_jT_{js}\g2_{sil}.
\end{align*}
\end{proof}

\medskip

\noindent
For a nearly $\G2$ structure we have $T_{ij}=\dfrac{\tau_0}{4}g_{ij}$ and $R_{ij} = \dfrac{3{\tau_0}^2}{8}g_{ij}$. Moreover from \cite[eq. $(4.18)$]{skflow},
\begin{align*}
(\pi_7(\Rm))_{jl}^j=-\del_l(\tr T)+\del_j(T_{lj})+T_{la}T_{jb}\g2_{abj}=0.
\end{align*}
Thus using the Weitzenb\"ock formula for $X$, $\del^*\del X_l=-\del_j\del_j X_l = (\Delta_d X)_l+R_{il}X_i$, we get the following

\begin{corollary}\label{cor:firstorderdiff}
Let $f\in C^{\infty}(M)$ and $X$ be a vector field on $M$ with a nearly $\G2$ structure $\g2$. Then
\begin{align}
\curl(\grad f)& =0, \label{eq:curlgrad1} \\
\Div(\curl X)&=0, \label{eq:divcurl1} \\ 
\curl(\curl X) &=\grad (\Div X)-\Delta X + \dfrac{3{\tau_0}^2}{8} X+\tau_0(\curl X), \label{eq:curlcurl1}\\
& = \Delta_d X + \grad(\Div X)+\tau_0 (\curl X). \label{curlcurl2}
\end{align}
\end{corollary}

\subsection{Identities for $2$-forms and $3$-forms}\label{subsec:23forms}

\noindent
In this subsection, we prove some identities for $2$-forms and $3$-forms on a manifold with a nearly $\G2$ structure. These identities will be used several times in the paper.

\begin{lemma}\label{lemma:iden2}
Let $(M, \g2)$ be a manifold with a $\G2$ structure. If $\beta = \beta_7+\beta_{14}$ is a $2$-form then
\begin{enumerate}[(1)]
\item $*(\beta\wedge \g2)=2\beta_7-\beta_{14}$.
\item $*(\beta\wedge \beta \wedge \g2)= 2|\beta_7|^2-|\beta_{14}|^2$.
\end{enumerate}
\end{lemma}
\begin{proof}
The identity in $(1)$ follows from \eqref{2formsdecomposition7} and \eqref{2formsdecomposition14}. For $(2)$ we note that for $7$-dimensional manifolds $*^2(\alpha)=\alpha$ for a $k$-form $\alpha$, so
\begin{align*}
\beta\wedge \beta \wedge \g2 = \beta \wedge *^2(\beta \wedge \g2) = \beta \wedge *(2\beta_7-\beta_{14})
\end{align*}
and the decomposition of $2$-forms is orthogonal.
\end{proof}

\begin{lemma}\label{lemma:3formiden1}
Let $(M, \g2)$ be a manifold with a $\G2$ structure. Let $\sigma = f\g2+\sigma_7+\sigma_{27}$ be a $3$-form on $M$ and let $\sigma_7=X \lrcorner \psi$ for some vector field $X$ on $M$. Then
\begin{enumerate}[(1)]
\item $*(\sigma \wedge \g2) = 4X$.
\item $*(\sigma \wedge \psi)=7f$.
\end{enumerate}
\end{lemma}

\begin{proof}
For $(1)$ we have
\begin{align}
*(\sigma \wedge \g2)&=*((f\g2 + \sigma_7 + \sigma_{27})\wedge \g2)  =*(\sigma_7\wedge \g2)=*((X\lrcorner *\g2)\wedge \g2) \nonumber \\
&= 4X
\end{align}
where we have used the fact that $\Omega^3_{1}\oplus \Omega^3_{27}$ lies in the kernel of wedge product with $\g2$ and \eqref{eq:impiden7} in the last equality. For $(2)$ we note that $\Omega^3_{7} \oplus \Omega^3_{27}$ lies in the kernel of wedge product with $\psi$ and $\g2\wedge \psi = 7\vol$. 
\end{proof}

\medskip

\noindent
Next, we explicitly derive the expressions for exterior derivative and the divergence of various components of $2$-forms and $3$-forms on a manifold with a nearly $\G2$ structure. Some of these identities are new, at least in the present form and we believe that they will be useful in other contexts as well. 

\begin{lemma}\label{lemma:iden3}
Suppose $(M, \g2)$ is a manifold with a nearly $\G2$ structure. Let $f\in C^{\infty}(M)$, $\beta\in \Omega^2_{14}$ and $X\in \Gamma(TM)$. Then
\begin{enumerate}[(1)]
\item $d(f\g2)=df\wedge \g2 + \tau_0f\psi$.
\item $d^*(f\g2) = -(df)\lrcorner \g2$.
\item $d\beta= \dfrac 14*(d^*\beta\wedge \g2)+\pi_{27}(d\beta)$.
\item $d(X\lrcorner \g2)=-\dfrac{3}{7}(d^*X)\g2+\dfrac{1}{2}*\Big( \Big(\dfrac{3\tau_0}{2}X-\curl X   \Big) \wedge \g2 \Big)+i_{\g2}\Big(\dfrac{1}{2}(\del_iX_j+\del_jX_i)+\dfrac{1}{7}(d^*X)g_{ij}\Big)$.
\item $d^*(X\lrcorner \g2)=\curl X$.
\item $d(X\lrcorner \psi) = -\dfrac{4}{7}d^* X \psi -\Big(\dfrac{1}{2}\curl X + \dfrac{\tau_0}{4}X \Big) \wedge \g2 - *i_{\g2}\Big(\dfrac 12(\del_iX_j+\del_jX_j)+\dfrac 17(d^*X)g_{ij}  \Big)$.
\end{enumerate}
\end{lemma}

\begin{proof}
We have
\begin{align*}
d(f\g2)&= df\wedge \g2 + fd\g2 \\
&= df \wedge \g2 + \tau_0f \psi
\end{align*}
where we have used \eqref{eq:ng2reln} which proves $(1)$. For part $(2)$ we compute
\begin{align*}
d^*(f\g2) = -*d*(f\g2) =-*d(f*\g2) = -*(df\wedge *\g2) = -df\lrcorner \g2 
\end{align*}
as $d\psi=0$. 

\noindent
We prove part $(3)$. Since $d\beta$ is a $3$-form so 
\begin{align}\label{eq:d214}
d\beta = \pi_1(d\beta)+\pi_7(d\beta)+\pi_{27}(d\beta).
\end{align}
We compute each term on the right hand side of \eqref{eq:d214}. We will repeatedly use the identities \eqref{eq:impiden1}--\eqref{eq:impiden10}. Suppose
\begin{align*}
\pi_1(d\beta) = a\g2
\end{align*}
for some $a\in C^{\infty}(M)$. Since $\Omega^3_7\oplus \Omega^3_{27}$ lies in the kernel of wedge product with $\psi$ and $\beta \wedge \psi =0$ for $\beta\in \Omega^2_{14}$, we have
\begin{align*}
0= d(\beta \wedge \psi) = d\beta \wedge \psi = \pi_1(d\beta)\wedge \psi = 7a\vol \\
\end{align*}
and hence
\begin{align*}
\pi_1(d\beta)=0.
\end{align*}
Suppose $\pi_7(d\beta)=X\lrcorner \psi$ for $X\in \Gamma(TM)$. Using \eqref{2formsdecomposition14} and Lemma \ref{lemma:3formiden1} $(1)$, we have
\begin{align*}
d^*\beta=*d*(\beta)=-*d(\beta \wedge \g2) = -*(d\beta \wedge \g2)-\tau_0*(\beta \wedge \psi) = -4X.    
\end{align*}

\noindent
Thus
\begin{align*}
\pi_7(d\beta)=-\frac 14d^*\beta\lrcorner \psi = \frac 14*(d^*\beta\wedge \g2),
\end{align*}
which proves $(3).$

\medskip

\noindent
Since $d(X\lrcorner \g2)$ is a $3$-form, so we will write
\begin{align}
d(X\lrcorner \g2)=\pi_1(d(X\lrcorner \g2))+\pi_7(d(X\lrcorner \g2))+\pi_{27}(d(X\lrcorner \g2))
\end{align}
and will calculate each term on the right hand side. As before, assume
\begin{align*}
\pi_1(d(X\lrcorner \g2))=a\g2
\end{align*}
for some $a\in C^{\infty}(M)$. Then
\begin{align*}
d((X\lrcorner \g2)\wedge \psi)=\pi_1(d(X\lrcorner \g2))\wedge \psi = 7a \vol 
\end{align*}
and hence
$7a=*d((X\lrcorner \g2)\wedge \psi)=*d(3*X)$.
So we get that
\begin{align*}
a=\frac 37*d*X=-\frac 37d^*X.
\end{align*}
Assume that
\begin{align*}
\pi_7(d(X\lrcorner \g2))=Y\lrcorner \psi
\end{align*}
for some $Y\in \Gamma(TM)$. Using the fact that $\Omega^3_1 \oplus \Omega^3_{27}$ lies in the kernel of wedge product with $\g2$ we get
\begin{align*}
d((X\lrcorner \g2)\wedge \g2)=d(X\lrcorner \g2)\wedge \g2 + (X\lrcorner \g2)\wedge d\g2 = \pi_7(d(X\lrcorner \g2))\wedge \g2 + \tau_0(X\lrcorner \g2)\wedge \psi =(Y\lrcorner \psi)\wedge \g2 + 3\tau_0*X.
\end{align*}
So we get
\begin{align*}
4*Y+3\tau_0*X=d((X\lrcorner \g2)\wedge \g2)=d(2*(X\lrcorner \g2))=2d(X\wedge \psi) = 2(dX)\wedge \psi
\end{align*}
which gives
\begin{align*}
Y=\frac 12\Big(*((dX)\wedge \psi)-\frac{3\tau_0}{2}X\Big)=\frac 12\Big(\curl X-\frac{3\tau_0}{2}X \Big)
\end{align*}
and hence
\begin{align*}
\pi_7(d(X\lrcorner \g2))=-\frac{1}{2}*\Big( \Big(\curl X - \frac{3\tau_0}{2}X \Big) \wedge \g2 \Big).
\end{align*}

\noindent
Recall the map $i_{\g2}$ from \eqref{eq:327express}. To calculate $\pi_{27}(d(X\lrcorner \g2))$ we have
\begin{align}
d(X\lrcorner \g2)_{imn}\g2_{jmn}+d(X\lrcorner \g2)_{jmn}\g2_{imn}&= \Big[\frac{-3}{7}(d^*X)\g2_{imn}+\frac 12 \Big( \Big( \curl X-\frac{3\tau_0}{2}X \Big) \lrcorner \psi \Big)_{imn}+i(h_0)_{imn} \Big]\g2_{jmn} \nonumber \\
& \quad + \Big[\frac{-3}{7}(d^*X)\g2_{jmn}+\frac 12 \Big( \Big( \curl X-\frac{3\tau_0}{2}X \Big) \lrcorner \psi\Big)_{jmn}+i(h_0)_{jmn} \Big]\g2_{imn} \nonumber \\
&=-\frac{36}{7}(d^*X)g_{ij}+8(h_0)_{ij} + \frac{1}{2}\Big( \curl X-\frac{3\tau_0}{2}X\Big)_s\psi_{simn}\g2_{jmn} \nonumber \\
& \quad +\Big( \curl X-\frac{3\tau_0}{2}X\Big)_s\psi_{sjmn}\g2_{imn} \nonumber \\
&= -\frac{36}{7}(d^*X)g_{ij}+8(h_0)_{ij}. \label{pi27dX}
\end{align}
We calculate the left hand side of \eqref{pi27dX}. We have

\noindent
\begin{align*}
d(X\lrcorner \g2)_{imn}\g2_{jmn}+d(X\lrcorner \g2)_{jmn}\g2_{imn}&= (\del_i(X_l\g2_{lmn})-\del_m(X_l\g2_{lin})+\del_n(X_l\g2_{lim}))\g2_{jmn}\\
& \quad + (\del_j(X_l\g2_{lmn})-\del_m(X_l\g2_{ljn})+\del_n(X_l\g2_{ljm}))\g2_{imn}\\
&=(\del_iX_l\g2_{lmn}-\del_mX_l\g2_{lin}+\del_nX_l\g2_{lim})\g2_{jmn}\\
& \quad +\dfrac{\tau_0}{4}(X_l\psi_{ilmn}-X_l\psi_{mlin}+X_l\psi_{nlim})\g2_{jmn}\\
&\quad +(\del_jX_l\g2_{lmn}-\del_mX_l\g2_{ljn}+\del_nX_l\g2_{ljm})\g2_{imn}\\
& \quad +\dfrac{\tau_0}{4}(X_l\psi_{jlmn}-X_l\psi_{mljn}+X_l\psi_{nljm})\g2_{imn}\\
\end{align*}
where we have used \eqref{eq:delphi} and \eqref{eq:ng2reln}. So
\begin{align*}
d(X\lrcorner \g2)_{imn}\g2_{jmn}+d(X\lrcorner \g2)_{jmn}\g2_{imn}&= (\del_iX_l\g2_{lmn}\g2_{jmn}-2\del_mX_l\g2_{lin}\g2_{jmn})\\
& \quad +\dfrac{\tau_0}{4}(X_l\psi_{ilmn}-X_l\psi_{mlin}+X_l\psi_{nlim})\g2_{jmn}\\
&\quad (\del_jX_l\g2_{lmn}\g2_{imn}-2\del_mX_l\g2_{ljn}\g2_{imn})\\
& \quad +\dfrac{\tau_0}{4}(X_l\psi_{jlmn}-X_l\psi_{mljn}+X_l\psi_{nljm})\g2_{imn}.
\end{align*}
We use the contraction identities \eqref{eq:phiphi1}, \eqref{eq:phiphi2} and \eqref{eq:phipsi1} to get
\begin{align*}
d(X\lrcorner \g2)_{imn}\g2_{jmn}+d(X\lrcorner \g2)_{jmn}\g2_{imn}&=4\del_iX_j+4\del_jX_i+4(\Div X)g_{ij}\\
& \quad + \dfrac{\tau_0}{4}(-4X_l\g2_{ilj}+4X_l\g2_{lij}+4X_l\g2_{lij})\\
& \quad +\dfrac{\tau_0}{4}(-4X_l\g2_{jli}+4X_l\g2_{lji}+4X_l\g2_{lji})\\
&= 4\del_iX_j+4\del_jX_i-4(d^*X)g_{ij}
\end{align*}
and so from \eqref{pi27dX} we get
\begin{align*}
-\dfrac{36}{7}(d^*X)g_{ij}+8(h_0)_{ij}&=4\del_iX_j+4\del_jX_i-4(d^*X)g_{ij}
\end{align*}
and thus
\begin{align*}
(h_0)_{ij}=\dfrac{1}{2}(\del_iX_j+\del_jX_i)+\dfrac{1}{7}(d^*X)g_{ij}
\end{align*}
which completes the proof of $(4)$. 

\medskip

\noindent
We obtain $(5)$ by
\begin{align*}
d^*(X \lrcorner \g2)=*d*(X\lrcorner \g2)=*d(X\wedge \psi)=*(dX\wedge \psi)=\curl X.
\end{align*}

\medskip

\noindent
To prove part $(6)$, we notice that since $d\psi=0$, $d(X\lrcorner \psi)=\cL_X \psi$ which is the image of $\cL_X\g2 = d(X\lrcorner \g2)+\tau_0 X\lrcorner \psi$ under the linearization of the map $\Theta$. We then use part (4) of the lemma and  \eqref{eq:linearizationmap1} to get part (6).

\end{proof}

\medskip

\noindent
We use the following important lemma on several occasions. 

\begin{lemma}\label{lemma:d3form}
Let $\g2$ be a nearly $\G2$ structure on $M$ and $\sigma$ be a $3$-form so that
\begin{align*}
\sigma = f\g2 + *(X\wedge \g2)+\eta
\end{align*}
where $\eta\in \Omega^3_{27}$ with $\eta=i_{\g2}(h)$ where $h$ is a symmetric traceless $2$-tensor. Then
\begin{align}
\pi_1(d\sigma) &= \Big(\tau_0f+\frac{4}{7}d^*X\Big)\psi,  \label{d3form1} \\
\pi_7(d\sigma) & = \Big(df +\frac{ \tau_0}{4}X+\frac 12\curl X-\frac 12 \Div h \Big)\wedge \g2,  \label{d3form2} \\
\pi_7(d^*\sigma) & = *\Big((-df+\tau_0X -\frac 23 \curl X - \frac 23 \Div h)\wedge \psi \Big). \label{d3form3} 
\end{align}
\end{lemma}

\begin{proof}
We note that $*\sigma = f\psi + (X\wedge \g2)+*\eta$ and since $\g2$ is a nearly $\G2$ structure hence
\begin{align}\label{d3formpf1}
d\sigma = df\wedge \g2 + \tau_0 f\psi + d*(X\wedge \g2) + d\eta
\end{align}
and 
\begin{align}\label{d3formpf2}
d^*\sigma = -*d* \sigma = -*(df\wedge \psi)-*d(X\wedge \g2)+d^*\eta .
\end{align}
Now $\pi_1(d\sigma)=\lambda \psi$ for some $\lambda\in C^{\infty}(M)$. We use Lemma \ref{lemma:iden3} $(6)$ to get,
\begin{align}
7\lambda &=\langle \lambda \psi, \psi \rangle = \langle \pi_1(d\sigma), \psi \rangle = \langle d\sigma, \psi \rangle \nonumber \\
&=\langle df\wedge \g2 + \tau_0 f\psi + d*(X\wedge \g2) + d\eta, \psi \rangle \nonumber  \\
&=\langle df\wedge \g2, \psi\rangle +7\tau_0f + 4d^*X + \langle d\eta, \psi \rangle.  \label{d3formpf3}
\end{align}
The first term on the right hand side of \eqref{d3formpf3} is $0$ as $df\wedge \g2 \in \Omega^4_7$ and $\psi \in \Omega^4_1$. The last term is also $0$ as from \eqref{eq:3formdecom3}
\begin{align*}
\langle d\eta, \psi\rangle \vol & = d\eta \wedge \g2 = d(\eta \wedge \g2)+\tau_0\eta \wedge \psi =0.
\end{align*}
Thus we get that
\begin{align*}
7\lambda = 7\tau_0f+4d^*X \ \ \ \ \ \ \ \ \implies \ \ \ \ \ \ \ \lambda = \tau_0f+\frac{4}{7}d^*X
\end{align*}
which gives \eqref{d3form1}.

\medskip

\noindent
To derive \eqref{d3form2} and \eqref{d3form3}, we will need to contract $\eta\in \Omega^3_{27}$ with $\g2$ on two indices and with $\psi$ on three indices. Using \eqref{eq:327express} and the contraction identities \eqref{eq:phiphi1} and \eqref{eq:phipsi2}, a short computation gives
\begin{align}
\eta_{ijk}\g2_{ajk}&=4h_{ia}, \label{d3formpf4} \\
\eta_{ijk}\psi_{aijk}&=0. \label{d3formpf5}
\end{align}
Suppose $\pi_7(d\sigma)=Y\wedge \g2$ for some $1$-form $Y$. Note that for an arbitrary $1$-form $Z$ we have
\begin{align*}
\langle Y\wedge \g2, Z\wedge \g2\rangle \vol &= Y\wedge \g2 \wedge *(Z\wedge \g2)\\
&= -Y\wedge \g2 \wedge(Z\lrcorner \psi) = 4Y\wedge *Z \\
&=4\langle Y, Z\rangle \vol .
\end{align*}
So from \eqref{d3formpf1} we have
\begin{align}
4\langle Y, Z\rangle&= \langle Y\wedge \g2, Z\wedge \g2\rangle = \langle \pi_7(d\sigma), Z\wedge \g2 \rangle = \langle d\sigma, Z\wedge \g2\rangle \nonumber \\
&= \langle df\wedge \g2 + \tau_0 f\psi + d*(X\wedge \g2) + d\eta , Z\wedge \g2 \rangle \nonumber \\
&= 4\langle df, Z\rangle + \langle d*(X\wedge \g2), Z\wedge \g2\rangle + \langle d\eta, Z\wedge \g2 \rangle . \label{d3formpf6}
\end{align}
We first use Lemma \ref{lemma:iden3} $(6)$ to calculate the second term on the right hand side of \eqref{d3formpf6}. We have

\begin{align*}
\langle d*(X\wedge \g2), Z\wedge \g2 \rangle &= \left \langle (\frac 12 \curl X + \frac{\tau_0}{4}X)\wedge \g2, Z\wedge \g2  \right \rangle = \langle 2\curl X+ \tau_0X, Z\rangle   
\end{align*}
\noindent
So in \eqref{d3formpf6}, we have
\begin{align}\label{d3formpf7}
4\langle Y, Z\rangle &= \langle 4df + \tau_0X+2\curl X, Z\rangle + \langle d\eta, Z\wedge \g2 \rangle .
\end{align}
We compute in local coordinates
\begin{align*}
\langle d\eta, Z\wedge \g2 \rangle & =\frac{1}{24}(d\eta)_{ijkl}(Z\wedge \g2)_{ijkl} \\
&=\frac{1}{24}(\del_i\eta_{jkl}-\del_j\eta_{ikl}+\del_k\eta_{ijl}-\del_l\eta_{ijk})(Z\wedge \g2)_{ijkl}\\
&=\frac{1}{6}(\del_i\eta_{jkl})(Z_i\g2_{jkl}-Z_j\g2_{ikl}-Z_k\g2_{jil}-Z_l\g2_{jki})\\
&= \frac 16(Z_i\del_i\eta_{jkl}\g2_{jkl}-3Z_j\del_i\eta_{jkl}\g2_{ikl})\\
&= \frac 16 (Z_i\del_i(\eta_{jkl}\g2_{jkl})-\frac{\tau_0}{4}Z_i\eta_{jkl}\psi_{ijkl}-3Z_j\del_i(\eta_{jkl}\g2_{ikl})+\frac{3\tau_0}{4}Z_j\eta_{jkl}\psi_{iikl}).
\end{align*}
We now use \eqref{d3formpf4}, \eqref{d3formpf5} and the fact that $h$ is traceless to get
\begin{align*}
\langle d\eta, Z\wedge \g2 \rangle &=\frac{1}{6}(Z_i\del_i(4\tr h)-0-3Z_j\del_i(4h_{ji}))\\
&=-2\langle \Div h, Z\rangle .
\end{align*}
Thus from \eqref{d3formpf7} we get
\begin{align*}
\langle Y, Z\rangle &= \Big \langle df +\frac{ \tau_0}{4}X+\frac 12\curl X-\frac 12 \Div h, Z \Big \rangle
\end{align*}
and since $Z$ is arbitrary, we get 
\begin{align*}
Y=df +\frac{ \tau_0}{4}X+\frac 12\curl X-\frac 12 \Div h
\end{align*}
which establishes \eqref{d3form2}.

\medskip

\noindent
Next, we see from \eqref{d3formpf2} and \eqref{2formsdecomposition7} that
\begin{align*}
d^*\sigma&=-*(df\wedge \psi)-*(dX\wedge \g2)+*\tau_0(X\wedge \psi)+d^*\eta \\
&=-*(df-\tau_0X\wedge \psi)-2\pi_7(dX)+\pi_{14}(dX)+d^*\eta
\end{align*}
which on using \eqref{eq:curlxdx} becomes
\begin{align}\label{d3formpf8}
d^*\sigma &= -*\Big (\Big(df-\tau_0X +\frac 23 \curl X\Big)   \wedge \psi \Big ) + \pi_{14}(dX) + d^*\eta.
\end{align}
Suppose $\pi_7(d^*\sigma)=*(W\wedge \psi)$ for some $1$-form $W$. For any $1$-form $Z$ we note that
\begin{align*}
\langle *(W\wedge \psi), *(Z\wedge \psi)\rangle \vol &= *(W\wedge \psi)\wedge Z\wedge \psi = *(W\wedge \psi)\wedge \psi \wedge Z =3*W\wedge Z = 3\langle W, Z\rangle \vol .
\end{align*}
Thus using \eqref{d3formpf8} and the orthogonality of the spaces $\Omega^2_7$ and $\Omega^2_{14}$, we have
\begin{align}
3\langle W, Z \rangle &= \langle *(W\wedge \psi), *(Z\wedge \psi)\rangle =\langle \pi_7(d^*\sigma), *(Z\wedge \psi)\rangle = \langle d^*\sigma, *(Z\wedge \psi)\rangle \nonumber \\
&= \langle -*((df-\tau_0X +\frac 23 \curl X)   \wedge \psi) + \pi_{14}(dX) + d^*\eta, *(Z\wedge \psi) \rangle \nonumber \\
&= \langle -3df+3\tau_0X-2\curl X, Z\rangle + \langle d^*\eta, *(Z\wedge \psi)\rangle . \label{d3formpf9} 
\end{align}
Using \eqref{d3formpf4} and \eqref{d3formpf5}, we compute the last term on the right hand side of \eqref{d3formpf9}, in local coordinates. We have
\begin{align*}
\langle d^*\eta, *(Z\wedge \psi)\rangle &= \langle d^*\eta, Z\lrcorner \g2 \rangle =\frac 12 (d^*\eta)_{ij}Z_m\g2_{mij}=-\frac 12 \del_p(\eta_{pij})Z_m\g2_{mij} \\
&=-\frac 12 Z_m(\del_p(\eta_{pij}\g2_{mij})-\frac{\tau_0}{4}\eta_{pij}\psi_{pmij}) \\
&=-\frac 12Z_m(4\del_ph_{pm}-0)=-2\langle \Div h, Z\rangle
\end{align*} 
and hence we get
\begin{align*}
\langle W, Z\rangle &= \Big \langle -df+\tau_0X -\frac 23 \curl X - \frac 23 \Div h, Z\Big \rangle .
\end{align*}
Since $Z$ is arbitrary we get
\begin{align*}
W=-df+\tau_0X -\frac 23 \curl X - \frac 23 \Div h
\end{align*}
which gives \eqref{d3form3}.
\end{proof}

\medskip

\noindent
\begin{remark}
The main point of the previous lemma is to exhibit a relation between $\pi_7(d\eta)$ and $\pi_7(d^*\eta)$. Such a relation is expected because of the form of the linearization of the map $\Theta$. More precisely, from \eqref{eq:linearizationmap1}, applying the linearization of $\Theta$ to Lie derivatives, we have $\pi_{27}(\cL_X\psi)=-*\pi_{27}(\cL_X\g2)$, $\langle d\eta, Z\wedge \g2\rangle_{L^2}=-\langle \eta, *\cL_X\psi \rangle_{L^2}$ and $\langle d^*\eta, Z\lrcorner \g2 \rangle_{L^2}=\langle \eta, \cL_X\g2\rangle_{L^2}$. The computations in local coordinates was done to relate $\pi_7(d\eta)$ and $\pi_7(d^*\eta)$ to the divergence of the symmetric $2$-tensor $h$. 
\end{remark}

\noindent
\begin{remark}
The previous lemma generalizes Proposition 2.17 from \cite{kargiannis-lotay} where the $\G2$ structure was assumed to be torsion-free $(\tau_0=0)$.
\end{remark}

\medskip

\noindent
We have the following corollary of Lemma \ref{lemma:d3form}.

\begin{corollary}\label{cor:d3form}
Let $\g2$ be a nearly $\G2$ structure and let $\eta\in \Omega^3_{27}$. Then 
\begin{enumerate}[(1)]
\item If $\eta$ is closed then $d^*\eta \in \Omega^2_{14}$.
\item If $\eta$ is co-closed then $d\eta \in \Omega^4_{27}$.
\end{enumerate}
\end{corollary}
\begin{proof}
In the notation of Lemma \ref{lemma:d3form} we get that $f=X=0$ and $\sigma=\eta$. Thus we get that $$\pi_7(d\eta)=0\ \ \ \ \ \ \iff\ \ \ \ \ \ \pi_7(d^*\eta)=0$$as from Lemma \ref{lemma:d3form}, both conditions are equivalent to $\Div h=0$. Now if $d\eta=0$ then $\pi_7(d^*\eta)=0$ and hence $d^*\eta \in \Omega^2_{14}$. If $d^*\eta=0$ then $\pi_7(d\eta)=0$. Also, since $f=X=0$, we know from \eqref{d3form1} that $\pi_1(d\eta)=0$. So $d\eta \in \Omega^4_{27}$.
\end{proof}

\medskip

\noindent
We also have a result similar to Lemma \ref{lemma:d3form} for $4$-forms which we state below. The proof follows from the proof of Lemma \ref{lemma:d3form} by taking $\zeta=*\sigma$ and noting that $*i_{\g2}(h)=-i_{\g2}(h)$. We expect that both Lemma \ref{lemma:d3form} and Lemma \ref{lemma:d4form} will be useful in other contexts as well.

\begin{lemma}\label{lemma:d4form}
Let $\g2$ be a nearly $\G2$ structure on $M$ and $\zeta$ be a $4$-form on $M$ so that
\begin{align*}
\zeta = f\psi +X\wedge \g2 + \zeta_0    
\end{align*}
where $X\in \Omega^1(M)$ and $\zeta_0\in \Omega^4_{27}$ with $\zeta_0=*i_{\g2}(h)$ where $h$ is a symmetric traceless $2$-tensor. Then
\begin{align}
\pi_7(d\zeta) &= W\wedge \psi \ \ \ \textup{where}\ \ \ W=df-\tau_0X+\frac 23 \curl X -\frac 23 \Div h, \label{d4form1}  \\
\pi_1(d^*\zeta) &= \Big( \tau_0f+\frac 47 d^*X \Big)\g2, \label{d4form2} \\
\pi_7(d^*\zeta) &= Y\lrcorner \psi \ \ \ \textup{where}\ \ \ Y= -df-\frac{1}{2}\curl X -\frac{\tau_0}{4}X-\frac 12 \Div h. \label{d4form3}
\end{align}
\end{lemma}

\medskip

\noindent
We get the following corollary.

\begin{corollary}\label{cor:d4form}
Let $\g2$ be a nearly $\G2$ structure on $M$ and let $\zeta_0\in \Omega^4_{27}$. Then
\begin{enumerate}
\item If $d\zeta_0=0$ then $d^*\zeta_0\in \Omega^3_{27}$.
\item If $d^*\zeta_0=0$ then $d\zeta_0\in \Omega^5_{14}$.
\end{enumerate}
\end{corollary}

\section{Hodge theory of nearly $\G2$ manifolds}\label{sec:hodgetheory}

\subsection{Dirac operators on nearly $\G2$ manifolds}\label{subsec:dirac}
We begin this section by defining the Dirac operator on $(M, \g2)$ with a nearly $\G2$ structure. We then define a \emph{modified} Dirac operator which is more suitable for our purposes. A $\G2$ structure on $M$ induces a spin structure, so $M$ admits an associated Dirac operator $\dirac$ on its spinor bundle $\slashed{S}(M)$. Since $\tau_0$ is constant, by rescaling the metric induced by the nearly $\G2$ structure, we can change the magnitude of $\tau_0$ and by changing the orientation, we can change its sign. In the later part of the paper, we study deformations of nearly $\G2$ structures through nearly $\G2$ structures ${\g2}_t$. Since the underlying metric of any nearly $\G2$ structure is positive Einstein, the family of metrics $g_t$ corresponding to ${\g2}_t$ will be positive Einstein and so by \cite[Corollary 2.12]{besse-book}, the scalar curvature $R_t$ is constant in $t$. Thus, by \eqref{nearlyscalar}, $\tau_0$ will be constant through the deformation. Henceforth, we will assume that $\tau_0=4$. \emph{The results of the paper do not depend on the value of $\tau_0$ chosen}. Recall the following definition from \textsection \ref{sec:intro} with $\tau_0=4$. 

\medskip

\begin{definition}
A spinor $\eta \in \Gamma(\slashed{S}(M))$ is called a \emph{Killing} spinor if for any $X\in \Gamma(TM)$ 
\begin{equation}\label{eq:killingspinor}
\del_X\eta = -\frac 12X\cdot \eta
\end{equation}
where $``\cdot"$ is the Clifford multiplication.
\end{definition}

\noindent
The real spinor bundle $\slashed{S}(M)$, as a $\G2$ representation, is isomorphic to $\Omega^0\oplus \Omega^1$, where the isomorphism is
\begin{align*}
(f, X) \longrightarrow f\cdot \eta + X\cdot \eta .
\end{align*}
For comparison with the Dirac-type operator which we define later, let us derive a formula for the Dirac operator $\dirac$ on a nearly $\G2$ manifold in terms of this isomorphism. 

\medskip

\noindent
A unit spinor $\eta$ on a nearly $\G2$ manifold $M$ satisfies \eqref{eq:killingspinor}. Thus
\begin{align*}
\dirac(f\eta)= \sum_{i=1}^7 e_i\cdot \del_{e_i}(f\eta)= \del f\cdot \eta + \frac 72 f\eta,
\end{align*}
where we have used the fact that $e_i\cdot e_i = -1$. Also,
\begin{align*}
\dirac (X\cdot \eta) &= \sum_{i=1}^7 e_i\cdot \del_{e_i}(X\cdot \eta) = \sum_{i=1}^7 (e_i\cdot \del_{e_i}X \cdot \eta + e_i\cdot X\cdot \del_{e_i}\eta) \\
&= (dX)\cdot \eta + (d^*X)\eta + \sum_{i=1}^7 e_i\cdot X\cdot \del_{e_{i}}\eta
\end{align*}
which on using $X\cdot e_i + e_i\cdot X = -2\langle X, e_i\rangle$ and \eqref{eq:killingspinor} becomes
\begin{align*}
\dirac (X\cdot \eta) &= (dX)\cdot \eta + (d^*X)\eta-\sum_{i=1}^7(X\cdot e_i\cdot \del_{e_i}\eta
+2\langle X, e_i\rangle \del_{e_i}\eta) \\
& = (dX)\cdot \eta + (d^*X)\eta -\frac{7}{2}X\cdot \eta + X\cdot \eta = (dX)\cdot \eta + (d^*X)\eta -\frac 52 X\cdot \eta .
\end{align*}
Thus we get
\begin{align}\label{diracdefn1}
\dirac (f\eta + X\cdot \eta) = \Big( \frac{7}{2}f+d^*X \Big)\eta + \Big(\del f+ dX -\frac{5}{2}X  \Big )\cdot \eta.
\end{align}
Now $dX$ is a $2$-form, hence $dX=\pi_7(dX)+\pi_{14}(dX)$. Since the Lie group $\G2$ preserves the nearly $\G2$ structure $\g2$, it preserves the real Killing spinor $\eta$ induced by $\g2$ and $\Omega^2_{14}(M)\cong \mathfrak{g}_2$, the Lie algebra of $\G2$, we have $\pi_{14}(dX)\cdot \eta=0$. Also, we know from \eqref{eq:curl2} that $\pi_7(dX)= \dfrac 13(\curl X)\lrcorner \g2$ and it follows from the definition of the Clifford multiplication, for instance as in \cite[\textsection 4.2]{karigiannis-notes}, that $(Y\lrcorner \g2)\cdot\eta = 3Y\cdot\eta$ for any $Y\in \Gamma(TM)$, we get that
\begin{align*}
\dirac (f,X)= \Big(\frac{7}{2}f+d^*X\Big)\eta + \Big(\del f+ \curl X-\frac 52 X  \Big)\cdot \eta
\end{align*}
which we will write as
\begin{align}\label{diracdefn3}
\dirac (f,X)= \Big(\frac{7}{2}f+d^*X, \del f+ \curl X-\frac 52 X  \Big).
\end{align}
\begin{definition}
The \emph{Dirac operator} $\dirac$ is a first-order differential operator on $\slashed{S}(M)$ defined as follows. Let $s=(f,X)\in \Gamma(\slashed{S}(M))$. Then
\begin{equation}\label{diracdefneqn}
\dirac (f,X)= \Big(\frac{7}{2}f+d^*X, \del f+ \curl X-\frac 52 X  \Big).
\end{equation}
The Dirac operator is formally self-adjoint, that is, $\dirac^*=\dirac$ and is also an elliptic operator.
\end{definition}

\medskip

\noindent
Consider the \emph{Dirac Laplacian} $\dirac^2 = \dirac^*\dirac$. We relate it to the Hodge Laplacian in the following

\begin{proposition}\label{prop:diractohodge}
Let $s=(f,X)$ be a section of the spinor bundle $\slashed{S}(M)$. Then
\begin{align}\label{diractohodgeeqn}
 \dirac^2(f, X)& = \Big(\Delta f + \frac{49}{4}f+d^*X,\  \Delta_dX+\curl X+\frac{25}{4}X + \del f \Big).
\end{align}
Thus $\dirac^2$ is equal to $\Delta_d$ up to lower order terms.
\end{proposition}

\noindent
\begin{proof}
Using Corollary \ref{cor:firstorderdiff}, we calculate
\begin{align*}
\dirac^2(f, X)&= \Big(\frac 72\Big(\frac 72 f +d^*X   \Big)+d^*\Big(\del f+\curl X-\frac 52 X\Big), d\Big(\frac 72f+d^*X \Big )+\curl\Big (\del f+\curl X-\frac 52 X\Big ) \\
& \qquad\ \ \ \ \ \ \ \ \ \ \ \ \ \ \ \ \ \ \ \ \qquad \qquad \ \ \ \ \ \ \ \ \ \ \ \ \ \ \ \ \ \ \ \  -\frac 52\Big (\del f+\curl X-\frac 52 X \Big )      \Big) \\
&=\Big(\Delta f + \frac{49}{4}f+d^*X,\  \Delta_dX+\curl X+\frac{25}{4}X + \del f \Big)
\end{align*}
which proves \eqref{diractohodgeeqn}.
\end{proof}

\medskip

\noindent
We need a modification of the Dirac operator defined above. The spinor bundle $\slashed{S}(M)$ is isomorphic to $\Omega^0_1\oplus \Omega^1_7$ and hence, via a $\G2$-equivariant isomorphism, it is also isomorphic to $\Omega^3_1\oplus \Omega^3_7$. We define the {\bf{modified Dirac operator}}, which we denote by $D$, as follows. Consider the map
\begin{align*}
D:&\ \Omega^0_1\oplus \Omega^1_7 \longrightarrow  \Omega^3_1\oplus \Omega^3_7 \\
& (f, X)\mapsto \frac 12*d(f\g2)+\pi_{1\oplus7}(d(X\lrcorner \g2)).
\end{align*}
Using Lemma \ref{lemma:iden3} $(4)$ with $\tau_0=4$, we get
\begin{align}\label{moddirac}
D(f,X) = \Big(2f - \frac 37d^*X, \frac{1}{2}df+6X-\curl X   \Big).    
\end{align}

\begin{remark}
We note that $D$ is defined in the same way as in \cite{kargiannis-lotay} where the authors denote the operator by $\check{\dirac}$.
\end{remark}

\medskip

\noindent
We find the kernel of $D$. Let $(f,X)\in \Omega^0\oplus \Omega^1$ be in the kernel of $D$. Then
\begin{align*}
2f-\frac{3}{7}d^* X &=0, \\
\frac{1}{2}df + 6X-\curl X&=0.
\end{align*}Taking $d^*$ of the second equation and using the first equation and equation \eqref{eq:divcurl1}, we get 
\begin{align*}
\Delta f = d^* df &= 2d^* \curl X - 12d^*X =-56 f.
\end{align*} 
Since $\Delta$ is a non-negative operator, $f=0$. For $X$, we have
\begin{align*}
d^*X=0 \ \ \ \ \ \ \ \textup{and}\ \ \ \ \ \ \ \curl X = 6X.    
\end{align*}
We want to prove that $X$ is a Killing vector field. Let $dX = Y \lrcorner \g2 + \pi_{14}(dX)$. Then 
\begin{align*}
dX\wedge\psi &= (Y \lrcorner \g2)\wedge\psi\\ 
&= 3* Y.
\end{align*} 
Therefore $\pi_7(dX)=\dfrac{1}{3}*(dX\wedge\psi)\lrcorner \g2 = \dfrac{1}{3}(\curl X)\lrcorner \g2 = 2 X\lrcorner \g2.$ From Lemma \ref{lemma:iden2} $(2)$, we have
\begin{align*}
\int_M dX\wedge dX \wedge \g2 &= 2\|2X\lrcorner \g2\|^2- \|\pi_{14}(dX)\|^2 = 8\langle X\lrcorner \g2,X\lrcorner \g2\rangle - \|\pi_{14}(dX)\|^2 \\
&=8\langle X,*((X\lrcorner \g2)\wedge \psi)\rangle - \|\pi_{14}(dX)\|^2 = 24\|X\|^2 - \|\pi_{14}(dX)\|^2.
\end{align*}
On the other hand, since $M$ is compact, using integration by parts we have \begin{align*}
\int_M dX\wedge dX \wedge \g2 &= \int_M X\wedge dX\wedge d\g2 = 4 \int_M X\wedge dX \wedge \psi = 4\int_M X\wedge (6* X) = 24\|X\|^2.
\end{align*}
Therefore, $\pi_{14}(dX) =0$ and $dX=\pi_7(dX)=2X\lrcorner \g2$. 
Now using Lemma \ref{lemma:iden3} $(4)$, along with the fact that $X\in \ker D$, i.e., $d^*X=0$ and $\curl X=6X$, we get 
\begin{align*}
0&=d(dX)=d(X\lrcorner \g2)= \frac 12 i_\g2 (\cL_X g),
\end{align*}
and hence $X$ is a Killing vector field. Therefore  $\ker D$ is isomorphic to the set of Killing vector fields $X$ such that $\curl X=6X$. We denote $\ker D$ by $\K$, that is, 
\begin{align}\label{Dkernelaux1}
\ker D = \K =\{X\in \Gamma(TM)\mid \cL_Xg=0 \ \textup{and}\  \curl X=6X\}.  
\end{align}

\begin{remark}
Note that the above can also be proved using the identity $\Delta X = d^* d X =-2d^* (X\lrcorner \g2) =12X$,  since $\Ric_g =6g$ for $\tau_0=4$. 
\end{remark}

 
\begin{remark} If we also want the vector field $X \in \K$ to preserve the $\G2$ structure, then  \begin{align*}
\cL_X\g2 = d(X\lrcorner \g2)+X\lrcorner \ d\g2 =4X\lrcorner \psi =0,
\end{align*}but since $\Omega^1\cong \Omega^4_7$, this implies $X=0$. Hence the only vector fields in $\K$ that preserve the $\G2$ structure are trivial. Note that when $\g2$ is a nearly $\G2$ structure of type-1, that is $\mathrm{dim}(K\slashed{S})=1$, every Killing vector field preserves the $\G2$ structure and hence $\K=\{0\}$. 
\end{remark}

\noindent
The motivation for defining the modified Dirac operator can be understood from the following.

\medskip

\noindent
Consider the following operator \begin{align*}
D^+ \ : \ \Omega^3_1 \oplus \Omega^5_7 &\to \Omega^4_{1\oplus 7}\\
(f\g2,X\wedge \psi) &\mapsto \pi_{1\oplus 7}(d(f\g2)+d^* (X\wedge \psi)).
\end{align*}
From previous calculations and Lemma \ref{lemma:iden3} we know that \begin{align*}
d(f\g2) &= df\wedge \g2 +4f\psi \in \Omega^4_{1\oplus 7}, \\
 \pi_{1\oplus 7}(d^*(X\wedge \psi)) &= \frac{3}{7}(d^*X)\psi+\frac{1}{2} \Big(\curl X - 6X \Big) \wedge \g2.
 \end{align*} Thus 
 \begin{align*}
 D^+ (f\g2,X\wedge \psi)&= \Big(4f +\frac{3}{7}(d^*X), df+\frac{1}{2} \Big(\curl X- 6X \Big)\Big).
 \end{align*} 
Doing a similar calculation as we did for $\ker D$, we observe that if $(f, X)\in \ker D^+$, then
\begin{align*}
\Delta f=-28f,\ \  \curl X=6X\ \ \ \implies \ \ f=0=d^*X\ \ \textup{hence}\ \  X\in \K    
\end{align*}
and so $\ker D^+= \ker D$. Since $ \Omega^3_1 \oplus \Omega^5_7 \cong \Omega^4_{1\oplus 7}$ and $D, D^+$ are self-adjoint operators, we have the following identification 
\begin{equation}\label{eq:omega47alternate}
\begin{aligned}
 \Omega^4_{1\oplus 7}&= \Ima D^+ \oplus \ker D^+= \Ima D^+ \oplus \ker D\\
 &= d\Omega^3_1 \oplus \pi_{1\oplus 7}(d^* \Omega^5_7) \oplus \{X\wedge \g2 | X\in \K\}.
 \end{aligned}
\end{equation}
This is used in the following important

\begin{proposition}\label{prop:4form} 
Let $(M,\g2,\psi)$ be a nearly $\G2$ manifold. Then the following holds. 
\begin{enumerate}
\item $\Omega^4 = \{X\wedge \g2 | X\in \K\}\oplus d\Omega^3_1 \oplus d^* \Omega^5_7 \oplus \Omega^4_{27}$.
\item We have an $L^2$-orthogonal decomposition $\Omega^4_{\textup{exact}} =\{X\wedge \g2 \mid  X\in \K\}\oplus d\Omega^3_{1}\oplus \Omega^4_{27,\textup{exact}}$.
\end{enumerate}
\end{proposition}
\begin{proof}
The first part of the proposition follows from the  decomposition of $\Omega^4_{1\oplus 7}$ in equation \eqref{eq:omega47alternate}.

\medskip

\noindent
For the second part we note that the space $d^* \Omega^5_{7}$ is $L^2$-orthogonal to exact $4$-forms. To prove the $L^2$-orthogonality of the remaining summands we proceed term by term. Let $X\in\K$, $d(f\g2)\in d\Omega^3_1$ and $\gamma \in \Omega^4_{27}$, such that $d\alpha = X\wedge \g2+d(f\g2)+\beta$ for some exact $4$-form $d\alpha$. Using the pointwise orthogonality of $\Omega^4_1$ and $\Omega^4_7$, we have
\begin{align*}
\langle X\wedge\g2,d(f\g2)\rangle_{L^2} &= \langle X\wedge\g2,df\wedge\g2+4f\psi\rangle_{L^2} \\
&= \langle X\wedge\g2,df\wedge\g2\rangle_{L^2} \\
&= 4\langle X, df\rangle_{L^2} = 4\langle d^* X , f\rangle_{L^2} = 0.
\end{align*}
Note that since $X\in \K$, Lemma \ref{lemma:iden3} $(6)$ implies that $X\wedge \g2= d\left(-\frac 14X\lrcorner \psi \right)$, and hence is exact. Thus, $\beta \in \Omega^4_{27, \textup{exact}}$. Let $\beta=d\alpha_0$. The $L^2$-orthogonality of $\Omega^4_{27}$ and $\Omega^4_1$, along with the identity $\g2\wedge * d\alpha=0$ implies 
\begin{align*}
\langle d\alpha_0,d(f\g2)\rangle_{L^2}&= \langle d\alpha_0, df\wedge\g2+4f\psi\rangle_{L^2} \\
&= \langle d\alpha_0, df\wedge\g2\rangle_{L^2} + \langle d\alpha_0 ,4f\psi\rangle_{L^2} =0.
\end{align*}
The orthogonality of $X\wedge\g2$ and $d\alpha_0$ follows from the $L^2$-orthogonality of $\Omega^4_7$ and $\Omega^4_{27}$.
\end{proof}

\medskip

\noindent
Thus, from the previous proposition, we know that any $4$-form $\alpha$ on a nearly $\G2$ manifold can be written as $\alpha = X\wedge\g2 + d(f\g2)+d^*(Y\wedge\psi) + \alpha_0$, for some $X\in \K,\ f\in C^{\infty}(M), Y\in \Gamma(TM)$ and $\alpha_0\in \Omega^4_{27}$. Since for $Y\in\K$, $d^*(Y\wedge\psi)= 0$, one can choose $Y\in \K^{\perp_{L^2}}$ in the previous proposition. 

\medskip

\noindent
Thus for every $4$-form $\alpha$ there exists unique $X\in\K, \ Y\in \K^{\perp_{L^2}} , \ f\in C^\infty(M)$ and $\alpha_0\in \Omega^4_{27}$ such that $$\alpha = X\wedge\g2 + d(f\g2)+d^*(Y\wedge\psi) + \alpha_0.$$

\medskip

\subsection{Harmonic $2$-forms and $3$-forms on nearly $\G2$ manifolds}\label{subsec:h23form}

The above decomposition of $4$-forms has a very useful application in determining the cohomology of nearly $\G2$ manifolds. We first note that since nearly $\G2$ manifolds are positive Einstein, it follows from Bochner formula and Hodge theory that any harmonic $1$-form is $0$ and hence $\mathcal{H}^1(M)=\mathcal{H}^6(M)=0$. The next two theorems describe the degree $3$, $4$ and degree $2$ and $5$ cohomology of a nearly $\G2$ manifold.

\begin{theorem}\label{thm:harmonic3and4form}
Let $(M,\g2,\psi)$ be  a complete nearly $\G2$ manifold. Then every harmonic $4$-form lies in $\Omega^4_{27}$. Equivalently, every harmonic $3$-form lies in $\Omega^3_{27}$.
\end{theorem}

\begin{proof}
Let $\alpha$ be a harmonic $4$-form that is $d\alpha = d^*\alpha=0$. 
From Proposition \ref{prop:4form} there exists $X\in \K, \ f\in C^\infty(M), \ Y\in \K^{\perp_{L^2}}$ and $ \alpha_0\in \Omega^4_{27}$ such that 
\begin{align*}
\alpha&= X\wedge\g2+d(f\g2)+d^*(Y\wedge\psi) + \alpha_0.
\end{align*}
Since $X\in\K$ and hence $6X=\curl X$, by Lemma \ref{lemma:iden3} $(6)$,  $d^*(X\wedge\g2)= 4X\lrcorner \psi \in \Omega^3_7$ and since $d(f\g2)=df\wedge\g2+4f\psi \in \Omega^4_{1\oplus 7}$, we have 
\begin{align*}
0=\langle \alpha,d(f\g2)\rangle_{L^2} &= \langle X\wedge\g2,d(f\g2)\rangle_{L^2} + \|d(f\g2)\|^2_{L^2}+\langle d^*(Y\wedge\psi), d(f\g2)\rangle_{L^2} +\langle \alpha_0,d(f\g2)\rangle_{L^2} \\
&= \langle d^*(X\wedge\g2),f\g2\rangle_{L^2} + \|d(f\g2)\|^2_{L^2} \\
&= \|d(f\g2)\|^2_{L^2}.
\end{align*}
Thus $d(f\g2)=0$ and hence $f=0$. 

\medskip

\noindent
Now, $0=d^*\alpha=d^*(X\wedge\g2)+d^*\alpha_0 = 4X\lrcorner \psi +d^*\alpha_0$. Using the identity, $(X\lrcorner\psi)\wedge\g2 = 4*X$ we have 
\begin{align*}
\|d^*\alpha_0\|^2_{L^2}&= 16\langle X\lrcorner \psi,X\lrcorner \psi\rangle_{L^2} \\
&=16\langle X, *((X\lrcorner\psi)\wedge\g2)\rangle_{L^2} = 64\|X\|^2_{L^2}.
\end{align*}
On the other hand, again by Lemma \ref{lemma:iden3} $(6)$ 
\begin{align*}
\|d^* \alpha_0\|^2_{L^2} &=\langle d^* \alpha_0, d^* \alpha_0\rangle_{L^2}\\
&= -4\langle d^* \alpha_0, X\lrcorner \psi\rangle_{L^2} \\
&= -4\langle \alpha_0,d(X\lrcorner \psi)\rangle_{L^2} = 16\langle \alpha_0, X\wedge\g2\rangle_{L^2} =0,
\end{align*}which implies $X=0$. So $\alpha = d^*(Y\wedge \psi)+\alpha_0$.

\medskip

\noindent
Since $d^* \alpha_0=0$, applying Corollary \ref{cor:d4form} on $\alpha_0$ implies $d\alpha_0\in \Omega^5_{14}$. This identity together with the closedness of $\alpha$ gives us \begin{align*}
0=\langle \alpha, d^*(Y\wedge\psi)\rangle_{L^2} &= \|d^*(Y\wedge\psi)\|^2_{L^2} + \langle \alpha_0, d^*(Y\wedge\psi)\rangle_{L^2} \\
&= \|d^*(Y\wedge\psi)\|^2_{L^2} + \langle d\alpha_0, Y\wedge\psi\rangle_{L^2} =  \|d^*(Y\wedge\psi)\|^2_{L^2} .
\end{align*}
as $Y\wedge \psi \in \Omega^5_{7}$. Hence $d^*(Y\wedge\psi)=0$ or equivalently $Y\in\K$, thus $Y=0$ which implies that $\alpha=\alpha_0$ which completes the proof of the theorem.
\end{proof}

\medskip

\noindent
We also describe the degree 2 (and hence degree 5) cohomology on nearly $\G2$ manifolds below. In combination with Theorem \ref{thm:harmonic3and4form}, this completely describes the cohomology of a nearly $\G2$ manifold. 

\begin{theorem}\label{lem:2coho}
Let $(M, \g2, \psi)$ be a complete nearly $\G2$ manifold with $\tau_0=4$. Let $\beta$ be a $2$-form with 
\begin{align*}
 \beta = \beta_7 + \beta_{14} = (X\lrcorner \g2) + \beta_{14}\ \ \textup{for\ some}\ X\in \Gamma(TM).   
\end{align*}
If $\beta$ is harmonic then $\beta\in \Omega^2_{14}$.
\end{theorem}

\begin{proof}
Suppose $\beta\in \Omega^2(M)$ is harmonic. Then $d\beta = d^*\beta=0$ and since $d$ and $d^*$ are linear, we have
\begin{align*}
d\beta_7+d\beta_{14}=0,\ \ \ \ \ \ d^*\beta_7+d^*\beta_{14}=0    
\end{align*}
which on using Lemma \ref{lemma:iden3} $(3)$, $(4)$ and $(5)$ imply
\begin{align*}
-\frac{3}{7}(d^*X)\g2 +\frac 12 * ((6X-\curl X)\wedge \g2) + i_{\g2}\Big ( \frac 12(\mathcal{L}_Xg)+\frac 17(d^*X)g    \Big )+\frac 14*(d^*\beta_{14} \wedge \g2)+\pi_{27}(d\beta_{14})=0     
\end{align*}
and 
\begin{align*}
d^*\beta_{14}=-\curl X.    
\end{align*}
Thus we get
\begin{align*}
-\frac{3}{7}(d^*X)\g2 +\frac 12 * ((6X-\curl X-\frac 12\curl X)\wedge \g2) + i_{\g2}\Big ( \frac 12(\mathcal{L}_Xg)+\frac 17(d^*X)g    \Big )+\pi_{27}(d\beta_{14})=0       
\end{align*}
and so
\begin{align}\label{2cohopf2}
d^*X=0,\ \ \ \curl X=4X\ \ \ \textup{and}\ \ \ \frac{1}{2}(\mathcal{L}_Xg)+\pi_{27}(d\beta_{14})=0. 
\end{align}
Now $\curl X=4X$, so taking $\curl$ of both sides and using \eqref{curlcurl2} with $d^*X=0$, we get
\begin{align*} \Delta_d X + 4\curl X = 4\curl X \ \ \ \implies \ \ \ \Delta_d X=0.   
\end{align*}
Thus $X$ is harmonic. Since nearly $\G2$ manifolds are positive Einstein, it follows from Bochner formula and Myers theorem that $X=0$. Hence $\beta=\beta_{14}\in \Omega^2_{14}$.
\end{proof}

\begin{remark}
Theorem \ref{lem:2coho} was also proved in a very different way in \cite[Remark 15]{ball-oliveira}. The theorem has the following interesting interpretation in the context of $\G2$-instantons on a nearly $\G2$ manifold, as already described in \cite[Corollary 14]{ball-oliveira}. For any $\alpha \in H^2(M, \bZ)$, by Theorem \ref{lem:2coho}, there is a unique $\G2$-instanton on a complex line bundle $L$ with $c_1(L)=\alpha$. 
\end{remark}

\begin{remark}
It was brought to the attention of the authors by Uwe Semmelmann and Paul-Andi Nagy that Theorem \ref{thm:harmonic3and4form} also follows from the description of nearly $\G2$ manifolds using Killing spinors which is based on an old result of Hijazi saying that the Clifford product of a harmonic form and a Killing spinor vanishes. We also describe degree $2$ cohomology by our methods. We believe that the methods and the identities described here, apart from being useful in other contexts, also have the potential to be extended to manifolds with \emph{any} $\G2$ structure (not necessarily nearly $\G2$) with suitable modifications. The authors are currently investigating this. 
\end{remark}

\section{Deformations of nearly $\G2$ structures}\label{sec:deform}

Let $(M,\g2,\psi)$ be a nearly $\G2$ manifold with a nearly $\G2$ structure $(\g2,\psi)$. We are interested in studying the deformation problem of $(\g2,\psi)$ in the space of nearly $\G2$ structures. The infinitesimal version of this problem was settled by Alexandrov and Semmelmann in \cite{deformg2}. We will obtain new proofs of some of their results using the results proved in the previous sections.




\medskip

\noindent
Let $\cP$ be the space of $\G2$ structures on $M$, that is, the set of all $(\g2, \psi)\in \Omega^3_+\times \Omega^4_+$ with $\Theta(\g2)=\psi$. Given a point $\fp = (\g2, \psi)\in \cP$ we define the tangent space $T_{\fp}\cP$.
\begin{lemma}\label{lemma:tspace}
The tangent space $T_{\fp}\cP$ is the set of all $(\xi, \eta)\in \Omega^3(M)\times \Omega^4(M)$ such that 
\begin{align*}
\xi &= 3f\g2 - X\lrcorner \psi + \gamma \\
\eta &= 4f\psi + X\wedge \g2 - *\gamma
\end{align*}
for some $f\in \Omega^0(M),\ X\in \Gamma(TM)$ and $\gamma\in \Omega^3_{27}$.
\end{lemma}
\begin{proof}
The proof immediately follows from equations \eqref{eq:linearizationmap1} and \eqref{eq:linearizationmap2} from Proposition \ref{prop:linearizationmap}.
\end{proof}

\subsection{Infinitesimal deformations}\label{subsec:infinitesimaldeform}

\noindent
We want to study deformations of a given nearly $\G2$ structure $\g2$ on a compact manifold $M$ by nearly $\G2$ structures $\g2_t$. We will only be interested in deformations of the nearly $\G2$ structures modulo the action of the group ${\bR}^* \times \textup{Diff}_0(M)$ where $\textup{Diff}_0(M)$ denotes the space of diffeomorphisms of $M$ which are isotopic to the identity. We first use Proposition \ref{prop:4form} to find a slice for the action of diffeomorphism group on $\cP$ which is used to find the space of infinitesimal nearly $\G2$ deformations, a result originally due Alexandrov--Semmelmann \cite{deformg2}.

\medskip

\noindent
For the purposes of doing analysis, we consider the H\"older space $\cP^{k,\alpha}$ of $\G2$ structures on $M$ such that $\g2$ and $\psi$ are of class $C^{k,\alpha}$, $k\geq 1,\ \alpha\in (0,1)$.   Let $\fp=(\g2, \psi)\in \cP^{k, \alpha}$ be a nearly $\G2$ structure such that the induced metric is not isometric to round $S^7$. Denote the orbit of $\fp$ under the action of $\textup{Diff}_0^{k+1, \alpha}(M)$ -- $C^{k+1, \alpha}$ diffeomorphisms isotopic to the identity, by $\cO_{\fp}$. The tangent space $T_{\fp}\cO_{\fp}$ is the space of Lie derivatives $\cL_X(\g2, \psi)$ for $X\in \Gamma(TM)$. We are interested in finding a complement $\cC$ of $T_{\fp}\cO_{\fp}$ in $T_{\fp}\cP$. 

\medskip

\noindent
If $(\xi, \eta)\in T_{\fp}\cP$ then using Proposition \ref{prop:4form} $(1)$, we can write
\begin{align*}
\eta = X\wedge \g2 + df\wedge \g2 + 4f\psi + d^*(Y\wedge \psi) + \eta_0
\end{align*}  
for unique $X\in \K \ f\in \Omega^0(M),\ Y\in \K^{\perp_{L^2}}$ and $\eta_0\in \Omega^4_{27}$. From Lemma \ref{lemma:iden3} $(4)$ we know that
\begin{align*}
*d*(Y\wedge \psi)&=-*d(Y\lrcorner \g2)=\frac{3}{7}(d^*Y)\psi - (3Y-\frac 12 \curl Y)\wedge \g2 -*i_{\g2}\Big(\frac 12 (\del_iY_j+\del_jY_i)+\frac 17(d^*Y)g_{ij}  \Big)
\end{align*}
and since $$\cL_Y \psi = d(Y\lrcorner \psi) =- \dfrac{4}{7}d^*Y \psi -\Big(\dfrac{1}{2}\curl Y + Y \Big) \wedge \g2 - *i_{\g2}\Big(\dfrac 12(\del_iY_j+\del_jY_j)+\dfrac 17(d^*Y)g_{ij}  \Big)$$ from Lemma \ref{lemma:iden3} $(6)$, we see that
\begin{align*}
d^*(Y\wedge \psi)&= -\frac 17(d^*Y)\psi + (\curl Y-2Y)\wedge \g2 - \cL_Y\psi.
\end{align*}
Thus up to an element in $T_{\fp}\cO_{\fp}$ we get that
\begin{align}\label{eq:complement1}
\eta = \Big(4f-\frac 17d^*Y \Big)\psi + (X+df+\curl Y -2Y)\wedge \g2 + \eta_0
\end{align}
and hence from Lemma \ref{lemma:tspace}
\begin{align}\label{eq:complement2}
\xi = (3f-\frac{3}{28}d^*Y)\g2 - (X+df+\curl Y -2Y)\lrcorner \psi -*\eta_0.
\end{align}


\medskip

\noindent
Now, if $X\in \K$ then from Lemma \ref{lemma:iden3} $(6)$ and $\curl X=6X$ we see that
\begin{align*}
\cL_{-\frac X4}\psi = d\left(-\frac X4\lrcorner \psi\right) = X\wedge \g2 
\end{align*}
and hence $$\eta = \cL_{-\frac X4}\psi + d(f\g2)+d^*(Y\wedge \psi)+\eta_0$$which implies that up to an element in $T_{\fp}\cO_{\fp}$ combined with the above observation, we can write 
\begin{align}\label{eq:complement1aux}
\eta = \Big(4f-\frac 17 d^*Y\Big)\psi + (df+\curl Y-2Y)\wedge \g2 +\eta_0    
\end{align}
which implies that 
\begin{align}\label{eq:complement2aux}
\xi=(3f-\frac{3}{28}d^*Y)\g2-(df+\curl Y-2Y)\lrcorner \psi - *\eta_0    
\end{align}
and hence we get a splitting $T_{\fp}\cP = T_{\fp}\cO_{\fp}\oplus \cC$ where $\cC\cong \Omega^0(M)\times  \K^{\perp_{L^2}}\times \Omega^4_{27}$ which consists of $(\xi, \eta)\in T_{\fp}\cP$ of the form \eqref{eq:complement2aux} and \eqref{eq:complement1aux} respectively. This gives a choice of slice. In fact, as discussed in \cite[pg. 49 \& Theorem 3.1.4]{nordstrom-thesis} we have 
\begin{proposition}\label{slicethm}
There exists an open neighbourhood $U$ of $\cC$ of the origin such that the ``exponentiation'' of $U$ is a slice for the action of $\textup{Diff}_0^{k+1, \alpha}(M)$ on a sufficiently small neighbourhood of $\fp\in \cP^{k, \alpha}$.
\end{proposition}

\medskip

\noindent
With this choice of slice we determine the infinitesimal deformations of the nearly $\G2$ structure $\fp$ which gives a new proof of a result of Alexandrov--Semmelmann \cite[Theorem 3.5]{deformg2}.

\begin{theorem}\label{thm:infidef}
Let $(M, \g2, \psi)$ be a complete nearly $\G2$ manifold, not isometric to the round $S^7$. Then the infinitesimal deformations of the nearly $\G2$ structure are in one to one correspondence with $(X, \xi_0)\in \K \times \Omega^3_{27}$ with
\begin{align}\label{eq:infidef}
*d\xi_0=-4\xi_0 \ \ \ \ \ \ \ \ \ \ \textup{and}\ \ \ \ \ \ \ \ \ \ \Delta X = 12X.
\end{align}
Hence $\xi_0$ is co-closed as well. Moreover, $\Delta_d \xi_0=16\xi_0$.
\end{theorem}
\begin{proof}
Let $(\xi, \eta)\in T_{\fp}\cP$ be an infinitesimal nearly $\G2$ deformation of a $\G2$ structure $\fp\in \cP$. So $\eta$ must be exact and hence from Proposition \ref{prop:4form} $(2)$, we can remove the $d^*(Y\wedge \psi)$ term, in which case \eqref{eq:complement1} and \eqref{eq:complement2} become
\begin{align}\label{eq:infidefpf1}
\eta = 4f\psi +(X+df)\wedge \g2 + \eta_0 \ \ \ \ \ \textup{and}\ \ \ \ \ \xi=3f\g2-(X+df)\lrcorner \psi -*\eta_0.
\end{align}
Moreover, for infinitesimal nearly $\G2$ deformations we must have $$d\xi=4\eta$$ and hence \eqref{eq:infidefpf1} implies
\begin{align*}
4f\psi + (4X+df)\wedge \g2 + 4\eta_0 + d((X+df)\lrcorner \psi) + d*\eta_0=0.
\end{align*}
Using Lemma \ref{lemma:iden3} $(6)$ for the fourth term above and taking inner product with $\psi$ gives
\begin{align*}
28f-4d^*(X+df)=0.
\end{align*}
But since $X\in \K \implies d^*X=0$ and hence we get $\Delta f = 7f$. Since $M$ is not isometric to round $S^7$, Obata's theorem then implies that $f=0$ and hence
\begin{align}\label{eq:infidefpf2}
\eta = X \wedge \g2 + \eta_0 \ \ \ \ \ \textup{and}\ \ \ \ \ \xi=-X\lrcorner \psi -*\eta_0
\end{align}
which proves the one to one correspondence between the infinitesimal nearly $\G2$ deformations and $\K\times \Omega^3_{27}$. Since $\Ric=6g$ and $X$ is a Killing vector field, we have $\Delta X=12 X$ which is the second part of \eqref{eq:infidef}.
Since $\eta_0$ is exact, $d\eta_0=0$. From \eqref{eq:infidefpf2} and the fact that $d\xi=4\eta$, we get
\begin{align*}
d*\eta_0=-4\eta_0    
\end{align*}
and hence
\begin{align*}
*d\xi_0=-4\xi_0.    
\end{align*}
Taking $d^*$ of both sides give $d^*\xi_0=0$. Moreover,
\begin{align*}
\Delta_d \xi_0 = d^*d\xi_0=-4d^**\xi_0 = -4*(d\xi_0)=16 \xi_0    
\end{align*}
which completes the proof of the theorem.
\end{proof}

\begin{remark}\label{auxrem}
From the computations for the proof of Proposition \ref{slicethm} we know that for $X\in \K$,
\begin{align*}
-4X\wedge \g2 = \cL_X\psi.    
\end{align*}
Thus, from Theorem \ref{thm:infidef} we see that the infinitesimal deformations of a nearly $\G2$ structure \emph{modulo diffeomorphisms} are in one-to-one correspondence with $\xi_0\in \Omega^3_{27}$ such that $*d\xi_0=-4\xi_0$.
\end{remark}

\medskip

\noindent
Motivated from the study of deformations of nearly K\"ahler $6$-manifolds by Foscolo \cite[\textsection 4]{foscolo} where he used observations of Hitchin \cite{hitchin}, we also want to interpret the nearly $\G2$ condition \eqref{eq:ng2reln} as the vanishing of a smooth map on the space of exact positive $4$-forms. Moreover, in order to study the second order deformations, it will be convenient to enlarge the space by introducing a vector field as an additional parameter which is natural when one considers the action of the diffeomorphism group. We elaborate on this below.

\medskip

\noindent
Let $\psi=d\alpha$ be an exact positive $4$-form, not necessarily satisfying the nearly $\G2$ condition. Let $\eta\in\Omega^4_{\exact}$ be the first order deformation of $\psi$. Hitchin in \cite{hitchin} defined a volume functional for exact $4$-form $\rho=d\gamma$ given by\begin{align*}
V(\rho)&= \int_M*\rho\wedge\rho,
\end{align*} and a quadratic form 
\begin{align*}
W(\rho, \rho')&=\int_M\gamma\wedge \rho'=\int_M \rho \wedge \gamma', 
\end{align*} 
where $\rho=d\gamma$ and $\rho'=d\gamma'$ are exact $4$-forms. We denote $W(\rho, \rho)$ by $W(\rho)$. When $M$ is compact, Hitchin proves \cite[Theorem 5]{hitchin} that stable $4$-forms (which is the same as a positive $4$-form in our case) $\eta\in\Omega^4_{\exact}(M)$ is a critical point of the volume functional $V$ subject to the constraint $W(\eta)=\textup{constant}$ if and only if $\eta$ defines a nearly $\G2$ structure.
The linearization of the volume functional at $\psi$ is given by
 \begin{align*}
dV(\eta)=\left.\frac{d}{dt}\right|_{t=0} V(\psi+t\eta) &= \int_M \g2\wedge \eta +\int_M *\eta\wedge\psi \\
&= 2\int_M\g2\wedge\eta.
\end{align*}
For the linearization of the quadratic form, suppose $\psi$ is exact with $\psi=d\alpha$. We use integration by parts to get \begin{align*}
dW(\eta)=\left.\frac{d}{dt}\right|_{t=0} W(\psi+t\eta) &=\int_M\alpha\wedge\eta+\int_M\gamma\wedge\psi\\
&= 2\int_M\alpha\wedge\eta.
\end{align*}
Let us define an energy functional $E$ on exact $4$-forms by 
\begin{align*}
E(\rho)&:= V(\rho)-4W(\rho).
\end{align*}
Then from above calculations
\begin{align*}
dE(\eta)&= \int_M(\g2-4\alpha)\wedge\eta = \int_Md((\g2-4\alpha)\wedge\gamma.
\end{align*}Therefore $\psi=d\alpha$ is a critical point of $E$ if and only if $dE(\eta)=0$ for every $\eta\in \Omega^4_{\exact}$ that is if and only if \begin{align*}
d\g2-4d\alpha &= d\g2-4\psi =0.
\end{align*}
Hence the critical points of the functional $E$ on $\Omega^4_{+,\exact}$ are nearly $\G2$ structures. Since the energy functional $E$ is diffeomorphism invariant, we can introduce an extra vector field, as $dE$ will vanish in the direction of Lie derivatives. Thus $\psi$ being a stable exact $4$-form can be given by the formula\begin{align*}
\psi&= \frac{1}{4}d(\g2-* d(Z\lrcorner \psi) )
\end{align*}for some $Z\in\Gamma(TM)$. We use these observations to write the nearly $\G2$ condition \eqref{eq:ng2reln} as the vanishing of a smooth map. Let us denote by $\widehat{P}$ the space of stable $3$ and stable, exact $4$-forms, i.e., $(\g2, \psi)\in \Omega^3_{+}\times \Omega^4_{+, \exact}$. We have the following

\begin{proposition}\label{prop:hitchinZ}
Suppose $(\g2, \psi)\in \widehat{\cP}$ satisfies 
\begin{align}\label{eq:extravf}
d\g2-4\psi = d*d(Z\lrcorner \psi)
\end{align}
for some vector field $Z$ and $*$ denotes the Hodge star with respect to a fixed background metric. Then $(\g2, \psi)$ is a nearly $\G2$ structure.
\end{proposition}

\begin{proof}
We will prove that $d(Z\lrcorner \psi)=0$. We note from \eqref{eq:impiden4} that $$(Z\lrcorner \psi)\wedge \psi=0$$So from \eqref{eq:extravf} we get that 
\begin{align*}
\|d(Z\lrcorner \psi)\|^2_{L^2}&= \langle d(Z\lrcorner \psi), d(Z\lrcorner \psi)\rangle_{L^2}\\
&=\langle (Z\lrcorner \psi), *d*d(Z\lrcorner\psi)\rangle_{L^2}\\
&=\langle(Z\lrcorner \psi), *(d\g2-4\psi)\rangle_{L^2}\\
&=\int_M(Z\lrcorner \psi)\wedge (d\g2-4\psi)=\int_M(Z\lrcorner \psi)\wedge d\g2
\end{align*}
Since $\g2$ is a $\G2$ structure and $d\psi=0$ from \eqref{eq:extravf}, we know from \eqref{torsionforms2} that $\tau_1=0$ and hence $d\g2$ has no component in $\Omega^4_7$. Thus $$\langle (Z\lrcorner \psi), *d\g2 \rangle =0$$which implies that
\begin{align*}
\|d(Z\lrcorner \psi)\|^2_{L^2} = \int_M(Z\lrcorner \psi)\wedge d\g2 =0
\end{align*}
which proves the proposition.
\end{proof}

\medskip

\noindent
Suppose we want to describe the local moduli space of nearly $\G2$ structures on a manifold $M$. If $\mathcal{NP}$ denotes the space of nearly $\G2$ structures on $M$ then the local moduli space is $\mathcal{M}=\mathcal{NP}/\textup{Diff}_0(M)$. A natural way to study this problem is to view the nearly $\G2$ structures on $M$ as the zero locus of an appropriate function, find the linearization of the function and prove its surjectivity, so that an Implicit Function Theorem argument describes $\mathcal{M}$.    

\medskip

\noindent
Now let $(\g2,\psi)$ be  a nearly $\G2$ structure on $M$. Let $U\subset \Omega^4_{+,\exact}$ be a small neighborhood of the  $4$-form $\psi$. Since the condition of being stable is open we can assume the existence of such a neighborhood. Thus for $\eta\in\Omega^4_{\exact}$ with sufficiently small norm with respect to the metric induced by $\g2$, $\tilde{\psi}= \psi+\eta$ is also a stable exact $4$-form. From Proposition \ref{prop:hitchinZ} the pair of stable forms $(\tilde{\g2},\tilde{\psi})$  defines a nearly $\G2$ structure if there exists a $Z\in\Gamma(TM)$ such that \begin{align*}
d\tilde{\g2}-4\tilde{\psi}=d*d(Z\lrcorner\tilde\psi).
\end{align*}
This condition is equivalent to the vanishing of the map \begin{align}
\Phi &: U\times \Gamma(TM)\to \Omega^4_{\exact} \nonumber \\
(\tilde{\psi},Z)&\mapsto d*\tilde{\psi}-4\tilde{\psi}-d*d(Z\lrcorner\tilde{\psi}).\label{zerolocus}
\end{align}
Thus, the nearly $\G2$ structures are the zero locus of the map $\Phi$ modulo diffeomorphisms.

\medskip

\noindent
Let $\xi$ be the dual of $\eta$ under the Hitchin's duality map $\Theta$ as in Proposition \ref{prop:linearizationmap}. The linearization of the map $\Phi$ at the point $(\psi,0)$ is given by \begin{align*}
d\xi-4\eta=d*d(Z\lrcorner\psi).
\end{align*}
Thus the obstructions on the first order deformations of the nearly $\G2$ structure $(\g2,\psi)$ are given by  $\Ima( D\Phi)$ which is characterized in the following proposition, whose proof is inspired from a similar theorem in the nearly K\"ahler case by Foscolo \cite[Proposition 4.5]{foscolo}.

\begin{proposition}\label{imageDphi} 
Let $(\g2,\psi)$ be a nearly $\G2$ structure and $(\xi,\eta)\in \Omega^3\times\Omega^4_{\exact}$ be a first order deformation in $\cP$. Then $\alpha\in\Omega^4_{\exact}$ lies in the image of $D\Phi$ if and only if 
\begin{align*}
\langle d^*\alpha-4*\alpha,\chi\rangle_{L^2} =0 
\end{align*} for all co-closed $\chi\in\Omega^3_{27}$ such that $\Delta \chi =16\chi$.
\end{proposition}

\begin{proof}
From Proposition \ref{prop:4form} $(2)$, there exists $X \in \K, f\in C^\infty (M)$ and $\eta_0\in\Omega^4_{27,\exact}$ such that
\begin{align*}
\eta&= X\wedge \g2 + d(f\g2)+\eta_0\\
&= d\left( -\frac 14X\lrcorner \psi +f\g2 \right )+ \eta_0
\end{align*}and from Lemma \ref{lemma:tspace}, the $3$-form
\begin{align*}
\xi&= 3f\g2 -(df+X)\lrcorner\psi -*\eta_0.
\end{align*}
By Proposition \ref{prop:4form}, $\alpha = Y\wedge\g2+d(h\g2)+\alpha_0$ for some $Y\in\K, h\in C^\infty(M),\alpha_0\in \Omega^4_{27,\exact}$. Such an $\alpha$ lies in the image of $D\Phi$ if
\begin{align*}
d\xi-4\eta -d*d(Z\lrcorner\psi) &= \alpha = d\Big(-\frac 14Y\lrcorner\psi +h\g2\Big )+\alpha_0.
\end{align*}
From Lemma \ref{lemma:iden3} (5) 
\begin{align*}
d^*(Z\wedge\psi) &= -*d(Z\lrcorner\g2)\\&= \dfrac{3}{7}(d^*Z)\psi-\dfrac{1}{2} \Big(6Z-\curl Z   \Big) \wedge \g2 -*i_{\g2}\Big(\dfrac{1}{2}(\del_iZ_j+\del_jZ_i)+\dfrac{1}{7}(d^*Z)g_{ij}\Big)
\end{align*} Comparing the last term in the above expression with that of $d(Z\lrcorner\psi)$ in Lemma \ref{lemma:iden3} we get \begin{align*}
d(Z\lrcorner\psi)&= \dfrac{1}{7}d^*Z\psi+(2Z-\curl Z)\wedge\g2 + d^*(Z\wedge\psi).
\end{align*} 
Using these expressions for $\xi,\eta$ and $d(Z\lrcorner\psi)$ we get \begin{align*}
d\xi-4\eta -d*d(Z\lrcorner\psi) &=d((-f-\dfrac{1}{7}d^*Z)\g2 -(df-2Z+\curl Z)\lrcorner\psi)-d*\eta_0-4\eta_0.
\end{align*}
Thus, for finding the $\Ima(D\Phi)$, we need to solve the equations
\begin{equation}\label{3equations}
 \begin{aligned}
f+\dfrac{1}{7}d^*Z&=-h\\
df-2Z+\curl Z&=\dfrac{1}{4}Y\\
-d*\eta_0-4\eta_0&=\alpha_0.
\end{aligned} 
\end{equation}
Let $\alpha_0=0$. Then by Implicit Function Theorem, a solution of the first pair of equations always exist if the operator \begin{align*}
\tilde{D} \ \colon \ \Omega^0\times\Omega^1 &\to \Omega^0\times\Omega^1\\
(f,Z)&\mapsto \Big(f+\dfrac{1}{7}d^*Z,df-2Z+\curl Z \Big)
\end{align*} is invertible in a small neighborhood of its zero locus. Since $\tilde{D}$ differs from the modified Dirac operator $D$ in \eqref{moddirac} only by self-adjoint zeroth-order term, it is self-adjoint and hence $\ker(\tilde{D})=\text{coker}(\tilde{D})$. A pair $(f,Z)$ is in the kernel of the operator $D$ if and only if \begin{align*}
f+\dfrac{1}{7}d^*Z&=0\\
df-2Z+\curl Z&=0.
\end{align*} Applying the operator $d^*$ on the second equation and using the fact that $d^*(\curl Z)=0$ gives 
\begin{align*}
0 &= d^*df -2d^*Z = d^*df+14f.
\end{align*}
Thus $f=0$ as $\Delta$ is a non-negative operator. The second equation then becomes \begin{align*}
\curl Z &= d^*(Z\lrcorner\g2)=*(dZ\wedge\psi)= 2Z
\end{align*}and Proposition \ref{prop:curlxdx} implies that $dZ= \dfrac{2}{3}Z\lrcorner \g2 + \pi_{14}(dZ)$. Using Lemma \ref{lemma:iden2} (2) we get that
\begin{align*}
\int_M dZ\wedge dZ \wedge\g2 &= \dfrac{8}{9}\|Z\lrcorner\g2\|^2 -\|\pi_{14}(dZ)\|^2\\
&= \dfrac{8}{3}\|Z\|^2 -\|\pi_{14}(dZ)\|^2.
\end{align*} 
On the other hand \begin{align*}
\int_M dZ\wedge dZ \wedge\g2 &= 4\int_M Z\wedge dZ \wedge\psi =8\|Z\|^2.
\end{align*} 
Combining these two equations we get $\dfrac{16}{3}\|Z\|^2 = -\|\pi_{14}(dZ)\|^2$ and hence $Z=0$ as well. Thus $\ker(\tilde{D})=\text{coker}(\tilde{D})=0$ and $\tilde{D}$ is invertible when $\alpha_0=0$ and we can always solve the first pair of equations in \eqref{3equations}. Thus there are no restrictions on $Y,h$ to be in the image of $D\Phi$. Moreover if $\alpha_0\neq 0$ satisfies the third equation in \eqref{3equations} then
 \begin{align*}
d^*\alpha_0 &= -d^*d*\eta_0-4d^*\eta_0,\\
*\alpha_0&= -d^*\eta_0-4*\eta_0
\end{align*}
which on using the fact that $*\eta_0$ is co-closed implies 
\begin{align*}
d^*\alpha_0-4*\alpha_0&= 16*\eta_0-d^*d*\eta_0= 16*\eta_0-\Delta_d*\eta_0.
\end{align*}
Thus $\alpha_0\in\Omega^4_{27,\exact}$ is a solution to the equation \eqref{3equations} (3) if and only if 
\begin{align*}
\langle d^*\alpha_0-4*\alpha_0,\xi_0\rangle_{L^2} =0 
\end{align*} for all co-closed $\xi_0\in\Omega^3_{27}$ such that $\Delta \xi =16\xi$. To complete the proof of the proposition we now only need to prove the $L^2$-orthogonality condition for $\alpha$. But observe that since $Y\in\K$ \begin{align*}
d^*\alpha&= d^*(Y\wedge\g2)+d^*d(h\g2) + d^*\alpha_0= -4Y\lrcorner\psi+d^*d(h\g2) + d^*\alpha_0,
\end{align*}
and so $d^*\alpha-4*\alpha= d^*d(h\g2)-4*d(h\g2)+d^*\alpha_0-4*\alpha_0$. Since $\xi$ is co-closed, from Corollary \ref{cor:d3form} $d\xi\in\Omega^4_{27}$ and
 \begin{align*}
\langle d^*d(h\g2),\xi\rangle_{L^2} &= \langle d(h\g2),d\xi\rangle_{L^2} =0.
\end{align*}Similarly \begin{align*}
\langle *d(h\g2),\xi\rangle_{L^2} &=\langle d^*(h\psi),\xi\rangle_{L^2} = \langle h\psi,d\xi\rangle_{L^2}  =0
\end{align*}
which completes the proof of the proposition.
\end{proof}

\noindent
\begin{remark}
Proposition \ref{imageDphi} puts a very strong restriction on the first order deformations of a nearly $\G2$ structure to be unobstructed.
\end{remark}

\subsection{Second-order deformations}\label{sec:secorderdeform}

\noindent
Following the work of Koiso \cite{koiso} on deformations of Einstein metrics and the work of Foscolo \cite{foscolo} on the second order deformations of nearly K\"ahler structures on $6$-manifolds, we define the notion of second order deformations of nearly $\G2$ structures. 

\begin{definition}\label{def-secorderdeform}
Given a nearly $\G2$ structure $(\g2_0, \psi_0)$ and an infinitesimal deformation $(\xi_1, \eta_1)$, a second order deformation of $(\g2_0, \psi_0)$ in the direction of $(\xi_1, \eta_1)$ is a pair $(\xi_2, \eta_2)\in \Omega^3\times \Omega^4$ such that 
\begin{align*}
\g2 = \g2_0+\epsilon \xi_1+\frac{\epsilon^2}{2}\xi_2,\ \ \ \ \ \ \psi=\psi_0+\epsilon \eta_1+\frac{\epsilon^2}{2}\eta_2
\end{align*}
is a nearly $\G2$ structure up to terms of order $O(\epsilon^2)$. An infinitesimal deformation $(\xi_1, \eta_1)$ is said to be \emph{obstructed to second order} if there exists no second-order deformation in its direction.
\end{definition}

\begin{remark}
Second order deformations are the same as the second derivative of a curve of nearly $\G2$ structures on a manifold $M$. 
\end{remark}
\begin{remark}
In a similar way, we can define higher order deformations of a nearly $\G2$ structure.
\end{remark}

\noindent
Following the discussion in the previous section and in particular Proposition~\ref{prop:hitchinZ}, in order to find second order deformations of a given nearly $\G2$ structure $(\g2_0, \psi_0)$, we look for formal power series defining positive \emph{exact} $4$-form
\begin{align*}
\psi_{\epsilon}=\psi_0+\epsilon \eta_1+\frac{\epsilon^2}{2}\eta_2 + \cdots
\end{align*}
where $\eta_i\in \Omega^4_{\exact}$ and a vector field 
\begin{align*}
Z_{\epsilon}=\epsilon Z_1+\frac{\epsilon^2}{2}Z_2 + \cdots
\end{align*}
which satisfy \eqref{eq:extravf}, that is
\begin{align}\label{eq:secorddeform1}
d\g2_{\epsilon}-4\psi_{\epsilon}=d*d(Z_{\epsilon}\lrcorner \psi_{\epsilon})
\end{align}
where $\g2_{\epsilon}$ is the dual of $\psi_{\epsilon}$. Note that the Hodge star $*$ is taken with respect to $\g2_{\epsilon}$.

\medskip

\noindent
Since we are interested in second order deformations, given an infinitesimal nearly $\G2$ deformation $(\xi_1, \eta_1)$, we set $Z_1=0$ and look for $\eta_2\in \Omega^4_{\exact}$ such that \eqref{eq:secorddeform1} is satisfied upto terms of $O(\epsilon^3)$. Explicitly, we write
\begin{align*}
\g2_{\epsilon}=\g2_0+\epsilon \xi_1 + \frac{\epsilon^2}{2}(\widehat{\eta_2}-Q_3(\eta_1))
\end{align*}
where $\widehat{\eta_2}$ denotes the linearization of Hitchin's duality map $\Theta$ for stable forms in Proposition \ref{prop:linearizationmap} and $Q_3(\eta_1)$ is the quadratic term of Hitchin's duality map. Since we want solutions to \eqref{eq:secorddeform1} up to second order, we look for $\eta_2$ such that
\begin{align}\label{eq:secorddeform2}
d\widehat{\eta_2}-4\eta_2 = d(Q_3(\eta_1))+d*d(Z_2\lrcorner \psi_0)
\end{align}
as $Z_1=0$ and $Z_2\lrcorner \psi_0$ is the only second order term in $Z_{\epsilon}\lrcorner \psi_{\epsilon}$. We know from Proposition \ref{imageDphi} that there are obstructions to finding second order deformations and hence in solving the above equation. We want to establish a one-to-one correspondence between second order deformations of a nearly $\G2$ structure and solutions to \eqref{eq:secorddeform2}. We do this in the following lemma. 

\begin{lemma}
Suppose $\eta_2$ is a solution of \eqref{eq:secorddeform2}. Then $d(Z_2\lrcorner \psi_0)=0$ and $(\widehat{\eta_2}-Q_3(\eta_1), \eta_2)$ defines a second-order deformation of $(\g2_0, \psi_0)$ in the direction of $(\xi_1, \eta_1)$ in the sense of Definition \ref{def-secorderdeform}. Conversely, every second order deformation $(\xi_2, \eta_2)$ is a solution to \eqref{eq:secorddeform2}.
\end{lemma} 

\begin{proof}
We start with 
\begin{align*}
\|d(Z_2\lrcorner \psi_0)\|^2_{L^2}&= \langle Z_2\lrcorner \psi_0, d^*d(Z_2\lrcorner \psi_0)\rangle_{L^2} \\
&= \langle Z_2\lrcorner \psi_0, *d*d(Z_2\lrcorner \psi_0)\rangle_{L^2}\\
&= \langle Z_2\lrcorner \psi_0, *(d\widehat{\eta_2}-4\eta_2-dQ_3(\eta_1))\rangle_{L^2}
\end{align*}
Since $d\psi_{\epsilon}=O(\epsilon^3)$, hence from \eqref{torsionforms1} and \eqref{torsionforms2} we see that for any vector field $Y$, $\int d\g2_{\epsilon}\wedge (Y\lrcorner \psi_{\epsilon})=O(\epsilon^3)$. Thus the terms which are $O(\epsilon^2)$ in $\int d\g2_{\epsilon}\wedge (Y\lrcorner \psi_{\epsilon})$ vanish, that is
\begin{align*}
\int d\g2_0 \wedge (Y\lrcorner \eta_2) + d\xi_1\wedge \eta_1 + d(\widehat{\eta_2}-Q_3(\eta_1))\wedge (Y\lrcorner \psi_{0})=0.
\end{align*}
Using the facts that $d\g2_0=4\psi_0$,  $d\xi_1\wedge \eta_1=0$, being an $8$-form on a seven dimensional manifold and $(Y\lrcorner \eta_2)\wedge \psi_0 = - (Y\lrcorner \psi_0)\wedge \eta_2$ we get that
\begin{align*}
\int d(\widehat{\eta_2}-Q_3(\eta_1))\wedge (Y\lrcorner \psi_0)-4\eta_2\wedge (Y\lrcorner \psi_0)=0
\end{align*}
Taking $Y=Z_2$ proves that $d(Z_2\lrcorner \psi_0)=0$. From \eqref{eq:secorddeform2} we get that $$d(\widehat{\eta_2}-Q_3(\eta_1))=4\eta_2$$which proves that $((\widehat{\eta_2}-Q_3(\eta_1), \eta_2))$ is a second-order deformation of $(\g2_0, \psi_0)$ in the direction of $(\xi_1, \eta_1)$ in the sense of Definition \ref{def-secorderdeform}. Conversely, suppose that $(\xi_2, \eta_2)$ is a second-order deformation of $(\g2_0, \psi_0)$. Then $d\xi_2=4\eta_2$. 
\end{proof}

\medskip

\noindent
From the previous proposition and Proposition \ref{imageDphi} we have that if $(\xi_2,\eta_2)$ is a second order deformation of the nearly $\G2$ structure $(\g2_0,\psi_0)$ in the sense of Definition \ref{def-secorderdeform} then \begin{align}\label{eq:seconddefQ}
    \langle d^*dQ_3(\eta_1)-4*dQ_3(\eta_1),\chi\rangle_{L^2} = 0 
\end{align} for all $\chi\in \Omega^3_{27}$ such that $d^*\chi=0, \Delta \chi = 16\chi$. The above equation simplifies to \begin{align*}
    \langle *Q_3(\eta_1),d\chi - 4*\chi\rangle_{L^2} &=0.
\end{align*}
Moreover, if $\chi$ is an infinitesimal deformation of $(\g2_0,\psi_0)$, then by Theorem \ref{thm:infidef} $\chi$ satisfies $d\chi=-4*\chi$ (which of course implies $d^*\chi=0 \ \textup{and}\ \Delta \chi=16 \chi$) and so the above equation is equivalent to 
\begin{align*}
     \langle Q_3(\eta_1),\chi\rangle_{L^2}&=0.
 \end{align*}

\section{Deformations on the Aloff-Wallach space}\label{sec:awspace}

In \cite[Prop. 8.3]{deformg2} Alexandrov--Semmelmann established that the space of infinitesimal deformations of the nearly $\G2$ structure on the  Aloff--Wallach space $X_{1,1}\cong\frac{\mathrm{SU}(3)\times \mathrm{SU(2)}}{\mathrm{SU}(2)\times \mathrm{U}(1)}$ is an eight dimensional space isomorphic to $\mathfrak{su}(3)$, the Lie algebra of $\mathrm{SU}(3)$. The rest of the paper is devoted to prove that these deformations are obstructed to second order.

\medskip

\noindent
The embedding of $\mathfrak{su}(2)$ and $\mathfrak{u}(1)$ in $\mathfrak{su}(3)\oplus\mathfrak{su}(2)$, which we denote by $\mathfrak{su}(2)_d$ and $\mathfrak{u}(1)$, following \cite{deformg2}, is given by  
\begin{align*}
\mathfrak{su}(2)_d&= \left\{ \Big( \begin{pmatrix}
a&0\\0&0
\end{pmatrix},a\Big) \mid a\in \mathfrak{su}(2) \right \}, \\ \mathfrak{u}(1)&=\text{span}\{C\}=\text{span} \left\{( \begin{pmatrix}
i&0&0\\0&i&0\\0&0&-2i
\end{pmatrix},0)\right \}.
\end{align*}
The Lie algebra $\mathfrak{su}(3)\oplus\mathfrak{su}(2)$ splits as 
\begin{align*}
    \mathfrak{su}(3)\oplus\mathfrak{su}(2)&=\mathfrak{su}(2)\oplus\mathfrak{u}(1)\oplus\mathfrak{m}
\end{align*} 
where $\m$ is the $7$-dimensional orthogonal complement of $\mathfrak{su}(2)\oplus\mathfrak{u}(1)$ with respect to $B$, the Killing form of $\mathfrak{su}(3)\oplus\mathfrak{su}(2)$. The normal nearly $\G2$ metric on $X_{1,1}$ is then given by $-\frac{3}{40}B$ where the constant $-\frac{3}{40}$ comes from our choice of $\tau_0=4$.
If we denote by $W$ the standard $2$-dimensional complex irreducible representation of $\mathrm{SU}(2)$ and by $F(k)$ the $1$-dimensional complex irreducible representation of $\mathrm{U}(1)$ with highest weight $k$, then as an $\mathrm{SU}(2)\times\mathrm{U}(1)$-representation 
\begin{align*}
    \mathfrak{su}(3)_\C &\cong S^2W\oplus WF(3)\oplus WF(-3)\oplus\C.
\end{align*}
Let $\{e_i\}_{i=1}^7$ be the basis of $\m$. If we define  $I=\begin{pmatrix}i&0\\0&-i\end{pmatrix}, J=\begin{pmatrix}0&-1\\1&0\end{pmatrix}\ \textup{and}\  K=\begin{pmatrix}0&i\\i&0\end{pmatrix}$, we have

\begin{align*}
  e_1 &:= \frac{1}{3}\left(\begin{pmatrix}
2I&0\\0&0
\end{pmatrix},-3I\right ), \ \ \ e_2:= \frac{1}{3}\left(\begin{pmatrix}
2J&0\\0&0
\end{pmatrix},-3J\right ), \ \ \ e_3:= \frac{1}{3}\left(\begin{pmatrix}
2K&0\\0&0
\end{pmatrix},-3K\right ),
\end{align*}
\begin{align*}
e_4&:= \frac{\sqrt{5}}{3}\left(\begin{pmatrix}
0&0&\sqrt{2}\\0&0&0\\-\sqrt{2}&0&0
\end{pmatrix},0\right ), \ \ \ e_5:= \frac{\sqrt{5}}{3}\left(\begin{pmatrix}
0&0&\sqrt{2}i\\0&0&0\\\sqrt{2}i&0&0
\end{pmatrix},0\right), \\ e_6&:= \frac{\sqrt{5}}{3}\left(\begin{pmatrix}
0&0&0\\0&0&\sqrt{2}\\0&-\sqrt{2}&0
\end{pmatrix},0 \right), \ \ \ e_7:=\frac{\sqrt{5}}{3} \left(\begin{pmatrix}
0&0&0\\0&0&\sqrt{2}i\\0&\sqrt{2}i&0
\end{pmatrix},0\right ).  
\end{align*} 
This basis is orthonormal with respect to the metric $g=-\frac{3}{40} B$. We use the shorthand $e^{i_1i_2\dots i_n}$ to denote the $n$-form $e^{i_1}\wedge e^{i_2}\wedge\dots\wedge e^{i_n}$. The nearly $\G2$ structure $\g2$ is given by \begin{align*}
 \g2&=e^{123}+e^{145}-e^{167}+e^{246}+e^{257}+e^{347}-e^{356}. 
\end{align*} 
As an $\textup{SU(2)}\times \textup{U(1)}$ representation, $\m_\C\cong S^2W\oplus WF(3)\oplus WF(-3)$ where 
\begin{align*}
    S^2W &= \textup{Span}\{e^1,e^2,e^3\},\ \ \ 
    WF(3)=\textup{Span}\{e^4-ie^5,e^6-ie^7\},\ \ \ WF(-3)=\textup{Span}\{e^4+ie^5,e^6+ie^7\}.
\end{align*}
By Theorem \ref{thm:infidef}, the space of first order deformations is given by $\{\xi\in\Omega^3_{27}\ | \  d\xi=-4*\xi\}$. In this example, it was found to be isomorphic to $\mathfrak{su}(3)$. As an $\mathrm{SU}(2)\times \mathrm{U}(1)$ representation, $\mathfrak{su}(3)$ is isomorphic to the span of $\{C,e_1,\dots,e_7\}$. The $\mathrm{SU}(2)\times \mathrm{U}(1)$-invariant homomorphism from $\mathfrak{su}(3)$ to $\Omega^3_{27}(X_{1,1})$ is given by $\text{Span}\{A\}$  where\begin{align*}
    A(C)&=\g2-7e^{123}, \ \ \ 
    A(e_1)=\frac{5}{3}(e^{145}+e^{167}),\\
    A(e_2)&=\frac{5}{3}(e^{245}+e^{267}),\ \ \
    A(e_3)=\frac{5}{3}(e^{345}+e^{367}),\\
    A(e_4)&=\frac{5}{9}(3e^{467}+e^{137}+e^{126}+e^{234}),\ \ \ 
     A(e_5)=\frac{5}{9}(3e^{567}+e^{235}-e^{136}+e^{127}),\\
     A(e_6)&=\frac{5}{9}(3e^{456}-e^{236}-e^{135}+e^{124}),\ \ \ 
     A(e_7)=\frac{5}{9}(3e^{457}-e^{237}+e^{125}+e^{134}).
    \end{align*}

\noindent
Let us fix an $\alpha\in\mathfrak{su}(3)$. The adjoint action of $h=(h_1,h_2)\in \mathrm{SU}(3)\times \mathrm{SU}(2)$ is given by \begin{align*}
  h^{-1}\alpha h&=  h_1^{-1}\alpha h_1=\begin{pmatrix}
    iv_1&x_1+ix_2&x_3+ix_4\\
    -x_1+ix_2&iv_2&x_5+ix_6\\
    -x_3+ix_4&-x_5+ix_6&-i(v_1+v_2)
    \end{pmatrix}
\end{align*}where $v_1,v_2,x_1,x_2,x_3,x_4,x_5,x_6$ are functions on $X_{1,1}$. 

\medskip

\noindent
The infinitesimal deformation $\xi_\alpha$  associated to $\alpha$ such that $d\xi_\alpha=-4*\xi_\alpha$ is given by 
\begin{align*}
    \xi_\alpha=\frac{v_1+v_2}{2} A(C)+\frac{v_1-v_2}{2}A(e_1)+\sum_{i=1}^6 x_iA(e_{i+1}).
\end{align*}

\medskip

\noindent
We can now compute the $4$-form $\eta_\alpha$ by using the relation $d\xi_\alpha=4\eta_\alpha=-4*\xi_\alpha$.
In order to show that the infinitesimal deformation $(\xi_\alpha,\eta_\alpha)$ associated to $\alpha$ is obstructed to second order, we need to compute the quadratic term $Q_3(\eta_\alpha)$ as discussed in equation \eqref{eq:seconddefQ} and find an element $\beta\in\mathfrak{su}(3)$ for which the $L^2$-inner product is non-zero.

\medskip

\noindent
To compute $Q_3(\eta_\alpha)$, one can use the algorithm for stable $4$-forms on manifolds with $\G2$ structures as discussed in \cite{hitchin}. Using the fact that $\xi_\alpha=-*\eta_\alpha$, one can easily show that for some non-zero constant $c_1$, $Q_3(\eta_\alpha)=c_1*Q_4(\xi_\alpha)$ where $Q_4(\xi_\alpha)$ is the quadratic term associated to $\xi_\alpha$. Thus, we will instead compute $Q_4(\xi_\alpha)$ and show that the inner product $\langle *Q_4(\xi_\alpha),\xi_\alpha\rangle_{L^2} \neq 0$ to prove obstructedness.

\medskip

\noindent
Consider $\g2_t=\g2+t\xi_\alpha$ to be a positive $3$-form for small $t$. We will denote the metric and the volume form induced by $\g2_t$ by $g_t$ and $\vol_t$ respectively. We have a Taylor series expansion 
\begin{align*}
    g_t&= g_0+tg_1 + t^2g_2+O(t^3)).
\end{align*}
Then one can define the symmetric bi-linear form $B_t$ by
\begin{align*}
    (B_t)_{ij}&=((e_i\lrcorner\g2_t)\wedge(e_j\lrcorner\g2_t)\wedge\g2_t)(e_1,\dots,e_7).
\end{align*} 
The zero order term of $B_t$, denoted by $B_0$ is given by $(B_0)_{ij}=((e_i\lrcorner\g2)\wedge(e_j\lrcorner\g2)\wedge\g2)(e_1,\dots,e_7)=\delta_{ij}$. Similarly, one can compute the linear term $(B_1)_{ij}=3((e_i\lrcorner\g2)\wedge(e_j\lrcorner\g2)\wedge\xi_\alpha)(e_1,\dots,e_7)$ and the quadratic term $(B_2)_{ij}=3((e_i\lrcorner\xi_\alpha)\wedge(e_j\lrcorner\xi_\alpha)\wedge\g2)(e_1,\dots,e_7)$. The metric is then defined using the relation (see for example, \cite{skflow})
\begin{align*}
    (B_t)_{ij}=6(g_t)_{ij}\sqrt{\det g_t}.
\end{align*}
The linear term in $\vol_t$ is proportional to $\g2\wedge\eta_\alpha+\psi\wedge\xi_\alpha$ and thus vanishes since $(\xi_\alpha,\eta_\alpha)\in \Omega^3_{27}\times\Omega^4_{27}$. Using the above formula we get that 
\begin{align*}
   \vol_t&=\sqrt{\det g_t}=1+At^2+O(t^3),
\end{align*}where $A$ is a quadratic polynomial in $v_1,v_2$ and $x_i,i=1..6$.
Using the Taylor series expansion of $g_t$ and $\sqrt{\det g_t}$, we can compute the Taylor series expansion of the Hodge star associated to $\g2_t$, $*_t=*_0+t*_1+t^2*_2+O(t^3)$. The Hodge star operator $*_t$ can be computed using the formula 
\begin{align*}
    *_t(e^{i_1i_2\dots i_k})&=\frac{\vol_t}{(7-k)!}g_t^{i_1j_1}\dots g_t^{i_kj_k}\epsilon_{j_1\dots j_7}e^{j_{k+1}\dots j_7}.
\end{align*}

\noindent
The quadratic term $Q_4(\xi_\alpha)$ is then given by
\begin{align*}
   Q_4(\xi_\alpha)&=*_2\g2+*_1\xi_\alpha.
\end{align*}

\medskip

\noindent
In the present case, for a general element $\alpha\in \mathfrak{su}(3)$, the quadratic term turns out to be very complicated and is not very enlightening. We define the cubic polynomial on $X_{1,1}$ by \begin{align*}
  f_\alpha([h])&= \langle *Q_4(\xi_\alpha), \xi_{\alpha}\rangle_{L^2}.
\end{align*}
Note that $f_\alpha$ is cubic in $\alpha$ since $Q_4(\xi_\alpha)$ and $\xi_\alpha$ are quadratic and linear in $\alpha$ respectively. This cubic polynomial can be lifted to a polynomial $P$ on the Lie group $\mathrm{SU}(3)\times\mathrm{SU}(2)$ by 
\begin{align*}
    f_\alpha([h])&=P(h^{-1}\alpha h).
\end{align*} 
This lift enables us to calculate the average of $P$ on $\mathrm{SU}(3)\times \mathrm{SU}(2)$ by using the Peter--Weyl theorem. To express the polynomial $P$ in a compact form, we will set $z_1=x_2+ix_1,z_2=x_4-ix_3,z_3=x_6+ix_5$. Then the cubic polynomial $P$ is given by
\begin{align}\label{polynomial}
\begin{split}
P(h^{-1}\alpha h)=&-\frac{97}{6}(v_1^2v_2+v_2^2v_1)+\frac{25}{9}\mathrm{Re}(z_1z_2z_3)-\frac{29}{6}(v_1^3+v_2^3)+\frac{5}{3}(v_1+v_2)|z_1|^2\\&+\frac{37}{18}(v_1|z_3|^3+v_2|z_2|^2)+\frac{31}{9}(v_1|z_2|^3+v_2|z_3|^2) 
    \end{split}
\end{align} 

\medskip

\noindent
The next step in proving obstructedness is to show that the average value of $P$ on $\mathrm{SU}(3)\times \mathrm{SU}(2)$ is non-zero. For this, we appeal to the Peter--Weyl theorem. The Peter--Weyl theorem states that for any compact Lie group $G$, we have 
\begin{align*}
    L^2(G)=\underset{V_\gamma \in G_{irr}}{\bigoplus}\mathrm{Hom}(V_\gamma,G)\otimes V_\gamma
\end{align*}
where $G_{irr}$ denotes the set of all non-isomorphic irreducible representations of $G$.

\medskip

\noindent
The cubic polynomial $P$ lies in the $\mathrm{SU}(3)\times\mathrm{SU}(2)$ representation $\textup{Sym}^3\mathfrak{su}(3)$. The average value of the function $P(g^{-1}\xi g)$ on $\textup{SU}(3)\times \mathrm{SU}(2)$ is the same as the average value of $R(h^{-1}\alpha h)$ where $R$ is the projection of $P$ to the invariant polynomials. This is because $(P-R)(h^{-1}\alpha h)$ lies in the non-trivial part of the Peter--Weyl decomposition and has an average value of zero. The  unique trivial sub-representation of $\textup{Sym}^3\mathfrak{su}(3)$ is generated by the determinant polynomial $i\det$ on $\mathfrak{su}(3)$ which is given by
\begin{align*}
    i\det(g^{-1}\alpha g)=&-(v_1v_2^2+v_2v_1^2)+(v_1+v_2)|z_1|^2 -(v_1|z_3|^2+v_2|z_2|^2)+2\mathrm{Re}(z_1z_2z_3).
\end{align*}
The average value of the polynomial $P$ can be computed by computing the inner product of $P$ with $i\det$. On $\mathfrak{su}(3)$, since the Killing form $B$ is non-degenarate, $g=-\frac{1}{12}B$ defines an inner product on $\mathfrak{su}(3)$. The inner product $g$ induces an inner product on $\textup{Sym}^3 \mathfrak{su}(3)$ in the natural way. All the computations that follow are done using $g$.

\medskip

\noindent
If $E_{ij}$ denotes the matrix with $1$ as the $(i,j)$-th entry and zero elsewhere, then the subspace of $\mathfrak{su}(3)$ generated by $\{E_{ij}-E_{ji}+i(E_{ij}+E_{ji})\mid  i,j=1,2,3, i\neq j\}$ is orthogonal to $\mathrm{Span}\{E_{11}-iE_{33},E_{22}-iE_{33}\}$. Moreover $E_{ij}-E_{ji}+i(E_{ij}+E_{ji}), \   i,j=1,2,3, i\neq j$ are also orthogonal to each other. Thus the only non-trivial terms occurring in the inner product of $P$ and $i\det$ are,
\begin{align*}
     &\|v_1^2v_2+v_2^2v_1\|^2=\frac{1}{3},\ \ \ \ \ \| \mathrm{Re}(z_1z_2z_3)\|^2 =\frac{2}{3},\ \ \ \ \ \langle v_1^3+v_2^3,v_1^2v_2+v_2^2v_1\rangle = -\frac{1}{4},  \\ &\|(v_1+v_2)|z_1|^2\|^2=1, \ \ \ 
      \|v_1|z_3|^2+v_2|z_2|^2\|^2= \frac{4}{3},\ \ \ 
     \langle v_1|z_2|^2+v_2|z_3|^2,v_1|z_3|^2+v_2|z_2|^2\rangle=-\frac{1}{3}.
     \end{align*}
     
\noindent     
From \eqref{polynomial} and the above computations we have that
\begin{align*}
    \langle P ,i\det \rangle&= \frac{97}{6}\left(\frac{1}{3}\right)+\frac{50}{9}\left(\frac{2}{3}\right)+\frac{29}{6}\left(-\frac{1}{4}\right)+\frac{5}{3}(1)-\frac{37}{18}\left(\frac{4}{3}\right)-\frac{31}{9}\left(-\frac{1}{3}\right)=\frac{191}{24}\neq 0.
\end{align*}

\medskip

\noindent
Thus we get the following theorem.

\begin{theorem}\label{thm:awobs}
The infinitesimal deformations of the homogeneous nearly $\G2$ structure on the Aloff--Wallach space $X_{1,1}\cong \frac{\mathrm{SU}(3)\times\mathrm{SU}(2)}{\mathrm{SU}(2)\times\mathrm{U}(1)}$ are all obstructed. 
\end{theorem}

\addcontentsline{toc}{section}{References}

\printbibliography

\vspace{0.8cm}

\noindent                               
Department of Pure Mathematics, University of Waterloo, Waterloo, ON, N2L 3G1, Canada.\\
\emph{e-mail address} (SD): \href{mailto:s2dwived@uwaterloo.ca}{s2dwived@uwaterloo.ca}\\
\emph{e-mail address} (RS): \href{mailto:r4singha@uwaterloo.ca}{r4singha@uwaterloo.ca}

\end{document}